\documentclass{article}
\usepackage[preprint]{arxiv_style}

\usepackage[utf8]{inputenc}
\usepackage{amssymb,amsmath,amsfonts, amsthm, dsfont, bbm}
\usepackage{mathrsfs}
\usepackage{graphicx}
\usepackage{enumerate}
\usepackage{color} 
\usepackage{float}

\theoremstyle{plain}

\newtheorem{theorem}{Theorem}[section]
\newtheorem{lemma}[theorem]{Lemma}
\newtheorem{corollary}[theorem]{Corollary}
\newtheorem{remark}[theorem]{Remark}
\newtheorem{proposition}[theorem]{Proposition}
\newtheorem{example}[theorem]{Example}

\newcommand{\R}{\mathbb{R}} 
\newcommand{\E}{\mathbb{E}}
\newcommand{\N}{\mathbb{N}}
\newcommand{\EE}{\mathbb{E}}
\newcommand{\PP}{\mathbb{P}}
\newcommand{\T}{\mathcal{T}}

\numberwithin{equation}{section}  

\begin{document}%\recd{}{}%Do not alter this line.

\title{Stein's Method for Distributions Modelling Competing and Complementary Risk Problems} 

\author{Anum Fatima\\ \institute{University of Oxford, UK \& Lahore College for Women University, Pakistan}\\
\email{fatima@stats.ox.ac.uk} \And  
Gesine Reinert \\
\institute{University of Oxford \& The Alan Turing Institute, London, UK}\\ \email{reinert@stats.ox.ac.uk}}

\maketitle

\begin{abstract}
Competing and Complementary risk (CCR) problems are often modelled using a class of distributions of the maximum, or minimum, of a random number of i.i.d.\,random variables; we call this class the CCR class of distributions. While the CCR distributions generally do not have an easy-to-calculate density or probability mass function, two special cases, namely the Poisson-exponential and the exponential geometric distributions, can easily be calculated. Hence, it is of interest to approximate CCR distributions with these simpler distributions. In this paper, we develop Stein's method for the CCR class of distributions to provide a general comparison approach to bound the distance between two CCR distributions and contrast this approach to bounds obtained using a Lindeberg argument. We detail the comparison for Poisson-exponential, and exponential-geometric distributions.
\end{abstract}


%\ams{60F05}{60E05}% insert the primary 2020 Maths Subject Classification number in the first bracket

\section{Introduction} \label{intro}
\pagenumbering{arabic}

Competing and complementary risk (CCR) problems typically focus on the failure of a system composed of multiple components, or a system with several, sometimes even a countably infinite number of risk factors that can cause its failure. Here we think of components and risk factors as having a random lifetime, denoted by a sequence of i.i.d.\,positive \,random variables $X_1, X_2, \ldots$; at the end of its lifetime the component fails, or the risk occurs. 

CCR systems are either sequential or parallel; in sequential systems, the whole system fails at the occurrence of the first among an unknown number $N$ of risk factors; risks compete with each other to cause failure. In this case, the observed failure time is the minimum of the lifetimes of the risk factors. In parallel systems, the system fails after all of an unknown number $N$ of risk factors occur, and the observed failure time is the maximum of lifetimes of these risk factors. In both setups, the lifetimes of the components or risks are modelled random, as is the number of risks $N$, which is assumed to be independent of the lifetimes. These two CCR settings are frequently studied together as $\max(X_1, \ldots, X_N) = - \min(-X_1, \ldots, -X_N)$; see \cite{basu_klein}.

CCR settings arise in many fields such as industrial reliability, demography, biomedical studies, public health, and actuarial sciences. For example, in a study of vertical transmission of HIV from an affected mother to the newborn child, several factors increase the risk of transmission, such as high maternal virus load and low birth weight. Exactly which factors determine the timing of transmission is an ongoing area of research. Even with a list of known risk factors, there are possibly many unknown risk factors involved. Hence  the problem is modelled as a CCR problem in \cite{tojeiro2014}, \cite{louzada2016}, \cite{cancho2011poisson} and \cite{louzada2012}. The number of successive failures of the air conditioning system of each member of a fleet of 13 Boeing 720 jet aeroplanes has been modelled as a CCR problem in \cite{KUS20074497}, with potential risk factors including defective components in the air conditioning system or errors committed during the production process. In \cite{ADAMIDIS199835} the period between successive coal-mining disasters is modelled as a CCR problem, with construction faults or human errors committed by inexperienced miners as examples of risks. The daily ozone concentrations in New York during May–September 1973 is treated as a CCR problem in \cite{jayakumar2021general}. 

We call the distributions used to model the lifetimes in CCR problems the {\it CCR family of distributions}. Examples from this family which are studied in the literature are the exponential-Poisson distribution \cite{KUS20074497}, the Poisson-exponential lifetime distribution \cite{cancho2011poisson}, the Weibull–Poisson distribution \cite{lu2012new}, the extended Weibull–Poisson distribution, and the extended generalised extreme value Poisson distribution \cite{ramos_dey2020}. \cite{tahir2016compounding} gives a detailed review of CCR distributions and proposes additional ones.

The probability mass function of a CCR distribution can be quite unwieldy. Hence, it is of interest to approximate a CCR distribution by a simpler CCR distribution such as the Poisson-exponential distribution or the exponential geometric distribution. For the latter distributions, the number of risks, or components, follows a zero-truncated Poisson or geometric distribution, respectively. To assess such approximations we develop Stein's method for CCR distributions. 

The seminal work of Charles Stein \cite{stein1986approximate} derives bounds on the approximation error for normal approximations. In \cite{chen1975poisson}, Stein's method is adapted to Poisson approximation, see also \cite{arratia1989two} and \cite{barbour1992poisson}. Generalisations to many other distributions and dimensions are available, see for example \cite{chen2011normal}, \cite{nourdin2012normal}, and \cite{mijoule2023stein}. The Stein operators in this paper are based on the density approach; see  \cite{ley2017stein}. A main difficulty in distributional comparison problems is the comparison of a discrete and a continuous distribution; related works on Stein's method include  \cite{goldstein2013stein}, which bounds the Wasserstein distance between a beta distribution and the distribution of the number of white balls in a P\'{o}lya-Eggenberger urn, and \cite{germain2023one}, where standardisations related to the ones we propose are used. Maxima of a fixed number of random variables have been treated with Stein's method in \cite{feidt2013stein}; Stein's method applied to a random sum of random variables can be found in \cite{pekoz_rollin_ross}. In this paper we propose a comparison of the distributions of a maximum (or minimum) of discrete and a maximum (or minimum) of continuous random variables when the number of these random variables is itself an independent random variable. In future work, the  Stein characterisations derived in our paper could be employed to construct Stein-based goodness of fit statistics as in \cite{betsch2019new}.

The remainder of this paper is organised as follows. Section \ref{sec:gen_Stein} gives a brief introduction to Stein's method using the density approach, and applies it to obtain a general representation for the CCR class of distributions through a Stein operator. Section \ref{sec:gen_comp} develops Stein's method for comparing CCR distributions; as an alternative approach, it also provides a comparison based on a Lindeberg argument. As a main illustration of our results Section \ref{sec:PE_dist} details Stein's method for the Poisson-exponential (PE) distribution and uses it to bound the total variation distance between a PE distribution and a distribution from the CCR family, and a distribution not from the CCR family. As a second illustration, in Section \ref{sec:EG} we develop Stein's method for the exponential-geometric distribution and give distributional comparisons in total variation distance. Section \ref{sec:patterns} gives a bound on the bounded Wasserstein distance between the distribution of maximum waiting time of sequence patterns in Bernoulli trials and a PE distribution. Proofs which are standard and would disturb the flow of the argument are postponed to Appendix \ref{sec:proof}.

\section{Stein's method for CCR distributions}  \label{sec:gen_Stein}

We use the notation $\R^+_{{>0}} = (0, \infty)$ and $\N = \{1, 2, 3, \ldots\}.$ The  backward  difference operator $\Delta^-$ operates on a function  $g: \R \rightarrow \R$ by $\Delta^-g(x) = g(x) - g(x-1)$. We note that $\Delta^- (gq)(y) = q(y)\Delta^- g(y) + g(y-1) \Delta^- q(y)$. For a function $h \in {\rm Lip}_b(1)$ its derivative is denoted by $h'$; it exists almost everywhere (by Rademacher's theorem).

For two probability distributions
$q$ and $p$ on  $\R^+_{{>0}}$, we seek bounds on distances of the form
\begin{equation} \label{distnce}
    d_{\mathcal{H}} (p,q) := \underset{h \in \mathcal{H}}{\sup} \, |\E h(X) - \E h(Z)|
\end{equation}
where $\mathcal{H}$ is a set of test functions,  $Z\sim q$ and $X \sim p$. The sets of functions $\mathcal{H}$ in \eqref{distnce} are, for the total variation distance ($d_{TV}$), $\mathcal{H} = \{\mathbb{I}[.\in A] : A \in \mathcal{B} ( \mathbb{R})\}$; and for the  bounded Wasserstein distance ($d_{BW}$), $$\mathcal{H} = {\rm Lip}_b(1):=\{h: \R^+_{{>0}} \rightarrow \mathbb{R}:  |h(x) - h(y)| \le |x-y| \; \mbox{ for all } \;  x,y\in\R; \|h\| \le 1 \}.$$ Here $\mathcal{B} ( \mathbb{R})$  denotes the Borel sets of $\mathbb{R}$, and $\| \cdot \|$ is the supremum norm in $\mathbb{R}$. We note here the alternative formulation for total variation distance based on Borel-measurable functions $h: \R \rightarrow \R$;  
\begin{eqnarray} \label{eq:dtv}
d_{TV}(p,q) = \frac{1}{2} \underset{\|h\| \le 1}{\sup}|\E h(X) - \E h(Z)|.
\end{eqnarray}

\subsection{Stein's method for distributional comparisons}

To obtain explicit bounds on the distance between a probability distribution $\mathcal{L}(X)$ of interest and a usually well-understood approximating distribution $\mathcal{L}_0(Z)$, often called the {\it target} distribution,  Stein's method connects the test function $h \in \mathcal{H}$ as in \eqref{distnce} to the distribution of interest through a  \textit{Stein Equation}
\begin{equation}\label{SE}
    h(x) - \E h(Z) = \mathcal{T}g(x).
\end{equation}
In \eqref{SE}, $\mathcal{T}$ is a \textit{Stein operator} for the distribution $\mathcal{L}_0 (Z)$, with an associated \textit{Stein class} $ \mathcal{F}(\mathcal{T})$ of functions such that $\E[\mathcal{T}g(Z)] = 0 \mbox{ for all } g \in \mathcal{F}(\mathcal{T}) $ if and only if $ Z \sim \mathcal{L}_0 (Z);$ thus; a Stein operator characterises the distribution. The distance \eqref{distnce} can then be bounded by $ d_{\mathcal{H}} (\mathcal{L}(X),\mathcal{L}_0(Z)) \le \underset{g \in \mathcal{F}(\mathcal{H})}{\sup}|\E\mathcal{T}g(X)| $ where $ \mathcal{F}(\mathcal{H}) = \{g_h | h \in \mathcal{H}\}$ is the set of solutions of the Stein equation \eqref{SE} for the set of test functions $h \in \mathcal{H}$. 

The Stein operator $\mathcal{T}$ for a probability distribution is not unique, see for example \cite{ley2017stein}. In this paper, we employ the so-called {\it density method} which uses the score function of a probability distribution to construct a Stein operator, called a {\it score Stein operator}. Following \cite{stein2004use, ley2013stein}, a score Stein operator for a continuous distribution with probability density function (pdf) $p$ and support $[a,b] \in \R$ acts on functions $g$, such that the derivative exists, as 
\begin{equation} \label{generalstein}
{\mathcal T}_p g (x)  = \frac{ (gp)' (x) }{p (x)};
\end{equation}
 here, $0/0 =0$. For differentiable $g$ and $p$, \eqref{generalstein} simplifies to ${\mathcal T}_p g (x) = g'(x) + g(x) \rho (x)$ where $\rho = p'/p$ is the score function of $p$. The {Stein} class ${\mathcal F}(\mathcal{T}_p)$ is the collection of functions $ g : \R \rightarrow \R$ such that $g(x) p(x)$ is differentiable with integrable derivative and $ \lim_{x \to {a,b} } g(x)p(x)=0 $. It is straightforward to see that for $p$, a pdf on ${[a,b] \subset \R}$,
\begin{align} \label{solution}
    g(x) = g_h(x) = \frac{ 1 }{p (x)} \int_0^x [h(t)-\E  h(X)] p(t) \,\mathrm{d}t
    \end{align}
solves \eqref{SE} for $h$, and that,  if $h$ is bounded, then $g \in \mathcal{F}(\mathcal{T}_p)$.

To use the Stein equation for a distributional comparison, let $X$ and $Y$ be two random variables with pdf $p_X$ and  $p_Y$, defined on the same probability space and with nested supports ${\rm{supp}}(p_Y) \subset {\rm{supp}}(p_X) =[a,b] \subset \R$, score functions $\rho_X$ and $\rho_Y$ and the corresponding score Stein operators $\mathcal T_X$ and $\mathcal T_Y$. Then
\begin{equation} \label{comparison}
    \E h(Y)-\E h(X) = \E g_{X, h}(Y)(\rho_Y(Y)-\rho_X(Y))
\end{equation}
where $g_{X, h}(x)$ is the solution of the Stein equation for $h$ and $\mathcal T_X$. Equation \eqref{comparison} is a special case of Equation 23 given in Section 2.5 of \cite{ley2017stein}.

For a discrete distribution with probability mass function (pmf) $q$ having  support $\mathcal{I}=[a,b] \subset {\mathbb{N}}$, the discrete backwards score function is $\frac{ \Delta^- q(y)}{q(y)}$, and, as in Remark 3.2 and Example 3.13 in \cite{ley2017stein}, a discrete  backward score Stein operator is 
\begin{equation}
{\mathcal T}_q g (y)  = 
 \Delta^- g(y) + \frac{ \Delta^- q(y)}{q(y)} g(y-1),  \quad   y \in \mathcal{I}. 
\end{equation}

In abuse of notation, we often refer to a Stein operator for the distribution of $X$ as the Stein operator for $X$, and similarly to the score function of the distribution of $X$ as the score function of $X$.  Further, if $X \sim p$ we also write $\mathcal{T}_X$ for $\mathcal{T}_p$.

\subsection{CCR distributions}

Let  $N \in \N$ be a random variable with finite second moment and let ${\bf{Y}}= (Y_1, Y_2, \ldots)$  be a sequence of positive i.i.d.\,random variables, with cdf $F_Y$, independent of $N$. Then 
\begin{equation} \label{eq:CCR_rv}
    W_{\alpha} =W_{\alpha}(N,  {\bf{Y}})=  \begin{cases}
        \min\{Y_1, Y_2, \ldots, Y_N\}, & \mbox{ if } \alpha = -1; \\
        \max\{Y_1, Y_2, \ldots, Y_N\}, & \mbox{ if }\alpha = 1,
    \end{cases}
\end{equation}
is called a CCR random variable with type indicator $\alpha \in \{-1, 1\}.$
Setting
\begin{equation} \label{U_minmax}
    U_Y^{\alpha}(\cdot) = \begin{cases}
    1- F_Y(\cdot), & \mbox{ if }\alpha = -1; \\
    F_Y(\cdot), & \mbox{ if }\alpha = 1,
\end{cases}
\end{equation}
and denoting $G_N(x) = \E\,  x^N$ the probability generating function of $N$,
$W_{\alpha} >0$ has cumulative distribution function (cdf)  
\begin{align} \label{CCR_class_cdf}
    F_{W_{\alpha}}(w) &= \PP (W_\alpha  \le w)\\
    &= \begin{cases}
    1-\sum_{n}\PP (N =n)U_Y^{\alpha}(w)^{n} = 1-(G_N\circ U_Y^{\alpha})(w), & \mbox{ if }
    \alpha = -1; \\
    \sum_{n}\PP (N =n)U_Y^{\alpha}(w)^{n} = (G_N\circ U_Y^{\alpha})(w), & \mbox{ if }
    \alpha = 1.
    \end{cases}
\end{align}
If the $Y_i$'s have a continuous distribution with pdf $f_Y$ then $W_\alpha$ has pdf
\begin{equation} \label{CCR_class}
    {f}_{W_{\alpha}}(w) = \sum_{n} \PP (N =n) n ({U_Y^{\alpha}}(w))^{n-1} {f}_Y(w) = f_Y(w)(G_N'\circ U_Y^{\alpha})(w), 
\end{equation}
see also \cite{tahir2016compounding}: here we used that $(U_Y^\alpha)'(w) = \alpha f_Y(w)$.
  
If the $Y_i$'s are discrete  with pmf $p_Y$ on $\N,$ the resulting random variable $W_{\alpha}$ has pmf $p_{W_{\alpha}}$ which can be expressed in terms of $G_N(\cdot)$ as 
\begin{equation} \label{CCR_pmf}
    p_{W_{\alpha}}(x) =  \alpha \left(G_N (U_Y^{\alpha}(x))  - G_N (U_Y^{\alpha}(x-1)) \right) = \alpha\Delta^- (G_N \circ U_Y^{\alpha}) (x).
\end{equation}

\begin{example}
\begin{enumerate} 
    \item If $N$ is a zero-truncated Poisson random variable with parameter $\theta$, then \eqref{CCR_class} simplifies to 
   $ f_{W_{\alpha}}(w) =  f_Y(w) \frac{\theta }{(1-e^{-\theta})} e^{-\theta(1-U_Y^{\alpha}(w))}, $ for $w > 0.$ This is the pdf of the extended Poisson family of distributions given in Equation 3 of \cite{ramos_dey2020}, where it is called the G-Poisson class of distributions \cite{tahir2016compounding}. 
    \item If $N$ is a Geometric($p$)  random variable, with pmf $\PP(N=n) = (1-p)^{(n-1)} p$ for $n \in \N  $, then \eqref{CCR_class} yields 
    $f_{W_{\alpha}}(w) =  f_Y(w) \frac{p}{(1-(1-p)U_Y^{\alpha}(w))^2},$ for $ w > 0,$ 
    which gives the distributions in equations 5.2 and 5.3 of \cite{marshallolkin}.
\end{enumerate}
\end{example} 

\subsection{Stein's method for the CCR class of distributions}

To obtain a Stein operator for the CCR random variable $W_{\alpha} = W_{\alpha}(N, {\bf Y})$  in \eqref{eq:CCR_rv}, we use the density method. First, we assume that the $Y_i$'s are continuous with differentiable pdf $f_Y$. From \eqref{CCR_class} the score function for the distribution of $W_{{\alpha}}$ is 
\begin{equation} \label{generalscore} 
   \rho_{W_{{\alpha}}}({w}) = {\alpha} {f}_Y({w}) \frac{(G_N'' \circ U_Y^{\alpha})(w)}{(G_N'\circ U_Y^{\alpha})({w})} + \rho_Y(w), 
\end{equation} 
with $\rho_Y = {f_Y'}/{f_Y}$ the score function of $Y$. Hence  $\mathcal{T}_{W_{\alpha}}$ given by
\begin{equation} \label{CCR_stein_op}
    \mathcal{T}_{W_{\alpha}}g(w) = g'(w)  + \rho_{W_{\alpha}}(w) g(w),
\end{equation}
for $g$ differentiable, is a Stein operator acting on the functions $g \in {\mathcal F}(\mathcal{T}_{W_{\alpha}})$. For  a test function $h \in \mathcal{H}$ the corresponding Stein equation is  
\begin{equation}\label{eq:steingen}
    g'({w})  + \rho_{{W_{{\alpha}}}}({w}) g({w}) = h({w}) - \E h ( {W_{\alpha}}).
\end{equation} 

Thus, for any random variable $X$, the distance $d_{\mathcal{H}}$ from \eqref{distnce}  between the distribution of $X$ and $W_{\alpha}$ can be bounded by bounding the expectation of the left-hand side of \eqref{eq:steingen}. For $Y_i$'s taking values in $\N$ with pmf $p_Y$, the {backward score function} is
\begin{equation} \label{eq:genscorediscrete}
    \rho_{W_{{\alpha}}}(w) = \frac{\Delta^- p_{W_{\alpha}}(w)}{p_{W_{\alpha}}(w)} = \frac{\Delta^-( \Delta^- (G_N \circ U_Y^{\alpha})  )(w) }{\Delta^- (G_N \circ U_Y^{\alpha})(w)},
\end{equation}
with corresponding discrete backward score Stein operator $\mathcal{T}_{W_{\alpha}}$ operating as
\begin{align} \label{CCR_stein_op_discrete}
\mathcal{T}_{W_{\alpha}} g (w)  = \Delta^- g(w) + \rho_{W_{\alpha}}(w) g(w-1).
\end{align}

\section{A general comparison approach} \label{sec:gen_comp}
 
To illustrate the use of Stein's method for CCR distributions we compare the distributions of two maxima or two minima of a random number of i.i.d.\,random variables, $W_{{\alpha}}(N, {\bf Y})$ and $W_{{\alpha}}(M, {\bf Z})$.

\begin{proposition} \label{CCR_comp}
Let $W_{\alpha_1}(N, {\bf Y})$ and $W_{\alpha_2}(M, {\bf Z})$, for $\alpha_1, \alpha_2 \in \{-1, 1\}$ be CCR random variables 
with pdf's $f_Y$ and $f_Z$ and score functions $\rho_Y$ and $\rho_Z$. Then for any test function $h$ such that  the  $W_{\alpha_1}(N, {\bf Y})$-Stein equation \eqref{eq:steingen} for $h$ has a solution $g=g_h$, 
\begin{eqnarray}
 \lefteqn{\left|\E h(W_{\alpha_1}(N, {\bf Y}))-\E h(W_{\alpha_2}(M, {\bf Z}))\right| = \Big| \EE g ( W) 
  \left(\alpha_1 f_Y(W) \frac{(G_N'' \circ U_Y^{\alpha_1})(W)}{(G_N' \circ U_Y^{\alpha_1})(W)} \right. } \nonumber 
 \\
 && \left.  \quad \quad \quad \quad \quad \quad  - \alpha_2 f_Z(W) \frac{(G_M'' \circ U_Z^{\alpha_2})(W)}{(G_M' \circ U_Z^{\alpha_2})(W)} \right)  
 + \EE g (W) ( \rho_Y - \rho_Z) (W) \Big|, \quad \quad \quad \quad
 \label{eq:gencomp}
\end{eqnarray}
where $W = W_{\alpha_{2}}(M, {\bf Z})$ and $U_{\cdot}^{\alpha}$ is as  in \eqref{U_minmax}.
 
When the $Y_i$'s are discrete with pmf $f_Y$ and the $Z_i's$ are discrete with pmf $f_Z$ then
\begin{eqnarray}
 \left|\E h(W_{\alpha_1}(N, {\bf Y}))-\E h(W_{\alpha_2}(M, {\bf Z}))\right|&= \, \, \left|  \EE g (W-1) 
 \left(f_Y(W) \frac{\Delta^-( \Delta^- (G_N \circ U_Y^{\alpha_1})  )(W) }{\Delta^- (G_N \circ U_Y^{\alpha_1})(W) } \right. \right. \nonumber \\
  & \left. \left.  -  f_Z(W) \frac{\Delta^-( \Delta^- (G_N \circ U_Z^{\alpha_2})  )(W) }{\Delta^- (G_N \circ U_Z^{\alpha_2})(W) } \right) \right|.
 \label{eq:gencomp_disc}
\end{eqnarray}
\end{proposition}
\begin{proof}
We substitute the score functions \eqref{generalscore} and \eqref{eq:genscorediscrete} in \eqref{comparison}; simplifying gives \eqref{eq:gencomp}, and \eqref{eq:gencomp_disc} respectively.
\end{proof}

For comparing a discrete and a continuous random variable we use the concept of {\it standardised Stein equations} as in \cite{ley2017stein} and \cite{germain2023one}. For a continuous random variable $W$ with score function $\rho_W$, and a differentiable  function $c: \R \rightarrow \R$, we define a $c-$standardised Stein operator ${\mathcal{T}_{W}^{(c)}}$ by
\begin{equation} \label{stand_stein_cont}
   \mathcal{T}_W^{(c)} (g(w))  = {\mathcal{T}_W}(cg)(w) = c(w) g'(w) + \left[ c(w) \rho_W (w) + c'(w) \right] g(w).
\end{equation} 
For a random variable $V \in \N$ with discrete backward score  function $\rho_V$ and a function $d: \N \rightarrow \R$ we define a $d-$standardised Stein operator for $V$ by 
\begin{eqnarray} \label{eq:discretestin}
     \lefteqn{{\mathcal{T}_V^{(d)}} (g(w)) = \mathcal{T}_{V}(dg)({w}) = 
    \Delta^- (dg) (w) + \rho_V(w) (dg)(w-1)} \\ \nonumber
    && \quad \quad = d(w-1)\Delta^-g(w) + g(w)\Delta^- d(w) +  d(w-1) g(w-1)\rho_V (w) .
\end{eqnarray}

For CCR random variables $W_{\alpha}(N, {\bf{Y}})$ and $W_{\alpha}(M, {\bf{Z}})$, when the $Y_i$'s are continuous on $\R^+_{{>0}} $ with differentiable pdf $f_y$ and the $Z_i$'s take values in $\N$, then we rescale $W_{\alpha}(M, {\bf Z})$ by dividing it by $n$,  to give $W_n = W_{\alpha, n} = \frac1n W_{\alpha}(M, {\bf Z})$.
If ${W_{\alpha}(M, {\bf Z})} \in \N$ has pmf $p$ and backward score function $\rho$, $W_n$ has pmf $ \PP (W_n = z) = p (nz)$ and backward score function  $\Tilde{\rho}_n(z) = \rho (nz)$. We note here that the ratio $\rho(nz) = \frac{ p(nz) - p(nz-1)}{p(nz)} $ is the score function of ${W_{\alpha}(M, {\bf Z})}$ evaluated at $nz$, which for $n \ne 1$ does not equal the score function of $W_n$. With $ \Delta^{-n} f(x) := f(x) - f(x-1/n)$
we obtain the  Stein operator $ {\mathcal T}_{n}^{(d)}$ given by 
\begin{equation} \label{eq:discretesteinop}
   {\mathcal T}_{n}^{(d)} g(z) = {\mathcal T}_n (dg) (z)  =
    \Delta^{-n} (dg)(z)  d(z) +\tilde{\rho}_n(z) (dg) \left( z - 1/n \right).
\end{equation}

\begin{proposition} \label{prop:comp}
Let ${W_{\alpha}(M, {\bf Z})}$ be a discrete CCR random variable with discrete backward score function ${\rho_W}$; for $n \in \N$ set $W_n = {W_{\alpha}(M, {\bf Z})}/n$ and $\Tilde{\rho}_n(z) = \rho_W(nz)$. Let ${W = W_{\alpha}(N, {\bf Y})}$ be a continuous CCR random variable with score function $\rho$. Let $h \in \mathcal{H}$  be a test function such that the ${\mathcal{L}}({W})$-Stein equation \eqref{SE} has solution $g=g_h$. Then, for any differentiable function $c: \R^+_{{>0}} \rightarrow \R$, and  for any function $d: \N \rightarrow \R$,
\begin{eqnarray}
    \lefteqn{|\E h(W_n) - \E h(W)|
    \le \left|
    \E [ n \Delta^{-n} (dg)(W_n) -  (cg)'(W_n) ]\right.} \nonumber\\
    && \left. \quad \quad \quad \quad + \E\left[ n {\tilde \rho}_n (W_n) (dg)\left( W_n - 1/n \right) - (cg)\left( W_n\right) \rho(W_n)\right] \right|. \label{disbound1}
\end{eqnarray}
\end{proposition}
\begin{proof}
For comparing the two distributions, for a given test function $h$ we have 
\begin{align}
    \E h(W_n) - \E h(W) = \E {\mathcal T}_{W} (cg)(W_n) = \E [ (cg)'(W_n) + (cg)(W_n) \rho(W_n)] \label{eq:step1}
\end{align}
with $g$ solving the continuous Stein equation \eqref{stand_stein_cont} for $h$.
Next, we note that for the Stein operator given in \eqref{eq:discretesteinop},
$\E {\mathcal T}_{n}^{(d)} (g) ({W_n}) = 0$ by construction. Hence, also $n  \E {\mathcal T}_n^{(d)} (g) (W_n) = 0.$ 
Thus, \eqref{eq:step1} yields  
\begin{eqnarray*}
  \lefteqn{\E h(W_n) - \E h(W) = \E \left[ (cg)'(W_n) + (cg)(W_n) \rho(W_n) \right.}\\
  && \left. \quad \quad \quad \quad \quad \quad \quad \quad - n \Delta^{-n} (dg)(W_n) - n {\tilde \rho}_n (W_n) (dg)\left( W_n - 1/n \right) \right] . \nonumber 
\end{eqnarray*}
 Re-arranging gives the assertion.
\end{proof}

The adaptation to other deterministic scaling functions would be straightforward; as in our examples we only scale by dividing by $n$, we concentrate on this case. As an aside, while such standardisations and scalings could perhaps be used for a ``bespoke derivative'' as in \cite{germain2023one},  the connection is not obvious.

An alternative comparison between CCR distributions can be achieved using a Lindeberg argument, to arrive at the following result.

\begin{proposition} \label{prop:lindeberg}
Let $W_{\alpha}(M, \bf{X})$ and $W_{\alpha}(N, \bf{E})$ be given in \eqref{eq:CCR_rv}. Then for functions $h \in {\rm Lip}_b(1)$, 
\begin{eqnarray}
   \lefteqn{\left| \E h ( W_{\alpha}(N, {\bf{E}}))   - \E h ({W_{\alpha}(M, {\bf{X}})}) \right|
    \le  || h'||  \sum_{i=1}^\infty  \E | E_i - X_i| \PP (M \ge i)} \\
  &&  \quad \quad \quad +  {\sum_{m,n=1}^\infty \PP (M=m, N=n) \left|\E h(W_{\alpha}(n, {\bf{E}})) -
     \E h(W_{\alpha}(m, {\bf{E}}))\right|.} \nonumber  
\end{eqnarray}
If $M$ and $N$ are identically distributed random variables, then     \begin{equation}
     |\E h(W_{\alpha}(N, {\bf{E}})) - \E  h (W_{\alpha}(N, {\bf{X}})) | \le  \| h'\|\sum_{i=1}^\infty  \E | E_i - X_i| \PP (N \ge i) .
     \label{prop2inequality}
\end{equation} 
\end{proposition}

\begin{proof} 
We employ a Lindeberg argument, as follows.  Defining $X_1, X_2, \ldots$ and $E_1, E_2, \ldots$ on the same probability space,
we have 
\begin{align}
\left| \E h ( W_{\alpha}(N, {\bf{E}}))   - \E h ({W_{\alpha}(M, {\bf{X}})}) \right| 
    \le &
    \left| \E h ( W_{\alpha}(N, {\bf{E}}))   - \E h ({W_{\alpha}(M, {\bf{E}})}) \right| \label{disBW_1}\\
& + \left| \E h ( W_{\alpha}(M, {\bf{E}}))   - \E h ({W_{\alpha}(M, {\bf{X}})}) \right|\label{disBW_2}.
\end{align}
To bound \eqref{disBW_1}, we simply note that
\begin{eqnarray}
   \lefteqn{|\E h(W_{\alpha}(N, {\bf{E}})) -
     \E h(W_{\alpha}(M, {\bf{E}}))|}
     \nonumber \\
     &\le& || h'|| \sum_{m,n=1}^\infty \PP (M=m, N=n) \E \left| W_{\alpha}(n, {\bf{E}}) -
     W_{\alpha}(m, {\bf{E}})\right| \label{disBW_1_exp}
\end{eqnarray}
Now to bound \eqref{disBW_2}, if $\alpha =1$, then
\begin{align}
    &{\left| \E h\left(\max\{E_1, \ldots, E_{M}\}\right) - \E h\left(\max\{X_1, \ldots, X_{M}\}\right) \right|} \nonumber \\
     \le &\sum_{m=1}^\infty \mathbb{P}(M=m) \sum_{i=1}^{m} \E   \left|h\left(\max\{ E_1,\ldots,E_i, X_{i+1}, \ldots, X_{m} \} \right) \right. \nonumber \\
    & \left. - h \left(\max\{E_1,\ldots,E_{i-1}, X_i, \ldots, X_{m}\}\right)\right| \nonumber \\
    \le &\sum_{m=1}^\infty \mathbb{P}(M=m) \sum_{i=1}^{m} \E\left| E_i - X_i\right| \|h'\| 
    = \sum_{i=1}^\infty \E\left| E_i - X_i\right|  \PP (M \ge i) \|h'\|. \label{disBW_2_exp}
\end{align}

Similarly, the minimum case ($\alpha =-1$) follows as $\min(X_i) \le \max(X_i) \le \sum X_i$.
Adding \eqref{disBW_1_exp} and \eqref{disBW_2_exp} and taking the supremum over all $h \in {\rm Lip}_b(1)$ gives the first assertion. The second assertion follows from coupling $M$ and $N$ so that $M=N$ a.s., in which case \eqref{disBW_1_exp} vanishes.
\end{proof}

Propositions \ref{CCR_comp}, \ref{prop:comp} and \ref{prop:lindeberg} complement each other; the first two yield bounds in $d_{TV}$ distance, for a general variable $W$, while the last result can be translated into a bound in $d_{BW}$ distance, for comparing CCR distributions; see Remark \ref{rem:lindeberg} for more details. We note that Proposition \ref{CCR_comp} can be used to compare a maximum and a minimum, whereas Proposition \ref{prop:comp} and \ref{prop:lindeberg} require the same value of $\alpha$.

\section{Application to the Poisson-exponential distribution} \label{sec:PE_dist}
 
Cancho et al. \cite{cancho2011poisson} introduced the Poisson-exponential distribution as a distribution of the maximum of $N$ independently and identically distributed exponential random variables from an infinite sequence ${\bf{E}}=  (E_1, E_2, \ldots,)$ so that $E_i \sim {\rm Exp}(\lambda)$ with parameter $\lambda$ (having mean $1/{\lambda}$), and $N$ following a zero-truncated Poisson distribution with parameter $\theta$, independently of  ${\bf{E}}$. This maximum has the PE distribution with parameters $\theta, \lambda >0$, denoted by $PE(\theta, \lambda)$, which has the differentiable pdf 
\begin{equation} \label{pdf}
p({w}| \theta, \lambda) 
= \frac{\theta \lambda e^{-\lambda {w} - \theta e^{-\lambda {w}}}}{1 - e^{-\theta}}; \quad \quad {w} > 0.
\end{equation}
To obtain a Stein operator for $PE(\theta, \lambda)$ we use \eqref{generalscore} with  
$G_N({u}) = \frac{e^{-\theta}}{1 - e^{-\theta}} \left( e^{\theta{u}} - 1 \right)$ and 
$\frac{G_N''({u})}{G_N'({u})} = \theta,$ yielding the  score function
\begin{align} \label{scorePE}
    \rho({w}) = \lambda (\theta e^{-\lambda {w}}-1);
\end{align}
\eqref{CCR_stein_op} gives
\begin{equation} \label{steinoperator}
{\mathcal T} f ({w}) = f'({w}) + \lambda (\theta e^{-\lambda {w}}-1) f({w}).
\end{equation}
For a  bounded test function $h 
\in \mathcal{H}$, the score $PE(\theta, \lambda)$-Stein equation is
\begin{equation} \label{steinPE}
    f'(w) + \lambda(\theta e^{-\lambda{w}}-1)f({w}) = h({w}) - \E h({W}),
\end{equation}
with the solution $f_h$, given by \eqref{solution}, satisfying the following bounds.

\begin{lemma} \label{allboundsforPE}
Let $ h: \R^+_{{>0}} \rightarrow \R  $ be bounded and $f$ denote the solution \eqref{solution} of the Stein equation \eqref{steinPE} for $h$. Let $\Tilde{h}(w) = h(w)-\E h(W)$ for $W \sim PE(\theta, \lambda)$. Then for all $w>0$ 
\begin{eqnarray} 
\label{helpful}
    \left|e^{-\lambda w}f(w) \right| &\le&\frac{ \|\Tilde{h}\|}{\theta \lambda } \left( 1-e^{-\theta + \theta e^{-\lambda w}}\right)\\
\label{boundexp}
     &\le& \frac{\|\Tilde{h}\|}{\theta \lambda}; \\
\label{bound2}
    |\lambda (\theta e^{-\lambda w}-1) f(w)| &\le& \|\Tilde{h}\|;\\
\label{boundfx}
    |f(w)| &\le& \frac{2\|\Tilde{h}\| }{\lambda}; \\
\label{bound1}
    |f'(w)|& \le & 2\|\Tilde{h}\|.
\end{eqnarray}
If in addition $h \in {\rm Lip}_b(1)$, then at all points $w$ at which $h'$ exists,
\begin{equation} \label{bound3}
    |f''(w)| \le \|h'\| + 2\lambda\theta \|\Tilde{h}\| + 3\lambda \|\Tilde{h}\|.
\end{equation}
\end{lemma}

\begin{proof} We write $p$ for the pdf of $PE(\theta, \lambda)$. To prove \eqref{helpful} and \eqref{boundexp}, we bound
\begin{eqnarray*}
 \left|e^{-\lambda w}f(w) \right|
    \le  \|\Tilde{h}\|  \frac{e^{-\lambda w}}{p(w)} \int_0^{w} p(t) \mathrm{d}t
     =   \|\Tilde{h}\| \frac{1-e^{-\theta(1-e^{-\lambda w})}}{\theta \lambda},
\end{eqnarray*}
and \eqref{helpful} follows. From $1-e^{-y} < 1 \ \forall \ y>0$, we get \eqref{boundexp}. 

\medskip
{\emph{Proof of \eqref{bound2}}}.
{\it Case 1: $ \theta e^{-\lambda w} - 1 > 0$.} In this case
$0<w< \frac{\ln \theta} {\lambda}$ and we have
\begin{align*}
|\lambda (\theta e^{-\lambda w}-1)f(w)|
\le \|\Tilde{h}\|\frac{\lambda (\theta e^{-\lambda w}-1)}{p(w)} \int_0^w p(t) \mathrm{d}t.
\end{align*}

As $p'(t) = \lambda (\theta e^{-\lambda t}-1) p(t)\ge \lambda (\theta e^{-\lambda w}-1) p(t)$ for $ 0< t < w < \frac{\ln \theta }{\lambda}$, it follows that  
\begin{align*}
 0 & < \frac{\lambda (\theta e^{-\lambda w}-1)}{p(w)} \int_0^w p(t) \mathrm{d}t \le \frac{1}{p(w)} \int_0^w p'(t) \mathrm{d}t =  (1-e^{-\theta (1 - e^{-\lambda w}) + \lambda w})\le 1.
\end{align*}
Hence we obtain the bound \eqref{bound2} for $0<w< \frac{\ln\theta} {\lambda}$. 

\medskip
{\it Case 2: $ \theta e^{-\lambda w} - 1 \le 0$.} 
In this case $w \ge \frac{\ln\theta}{\lambda}$ and
$$ 0 < \lambda (1-\theta e^{-\lambda w}) p(t)
< \lambda (1-\theta e^{-\lambda t}) p(t) = p'(t).
$$ Using \eqref{solution} gives
\begin{align*}
    |\lambda (\theta e^{-\lambda w}-1) f(w)|  \le \|\Tilde{h}\| \frac{ \lambda (1-\theta e^{-\lambda w}) }{p (w)} \int_{w}^{\infty} p(t)\mathrm{d}t
    \le \|\Tilde{h}\|
    \frac{1}{p (w)} \int_{w}^{\infty} p'(t)\mathrm{d}t
    \le \|\Tilde{h}\|.
\end{align*} 
Hence the bound \eqref{bound2} follows for all $w>0$.

\medskip
{\emph{Proof of \eqref{boundfx}, \eqref{bound1} and \eqref{bound3}}}.
As $ \lambda(\theta e^{-\lambda w} -1)f(w) = \lambda\theta e^{-\lambda w} f(w) - \lambda f(w),$ the triangle inequality along with \eqref{bound2} and \eqref{boundexp} gives \eqref{boundfx}. To show \eqref{bound1}, using the triangle inequality, from \eqref{steinPE} we obtain 
$|f'({w})| \le |h({w})-\E h({W})|+|\lambda (\theta e^{-\lambda {w}}-1)f({w})|$
and using \eqref{bound2} yields the bound \eqref{bound1}. 

Now  for $h$  differentiable at $w$, taking the first order derivative  in \eqref{steinPE} gives
 $
     |f''(w)| \le |h'(w)| + |\lambda(\theta e^{-\lambda w} -1)f'(w)| + |\theta \lambda^2 e^{-\lambda w}f(w)|.
 $
Using \eqref{boundexp} and \eqref{bound1} we obtain the bound \eqref{bound3} through
\begin{align*}
    |f''(w)| &\le \|h'\| + \lambda|\theta e^{-\lambda w} -1| 2\|\Tilde{h}\| + \theta \lambda^2 \frac{\|\Tilde{h}\|}{\theta \lambda} \le \|h'\| + 2\lambda\theta e^{-\lambda w} \|\Tilde{h}\| + 2\lambda \|\Tilde{h}\| +  \lambda \|\Tilde{h}\|.
\end{align*}
This completes the proof.
\end{proof}

\begin{remark} 
  If $\theta \rightarrow 0$,  the PE  distribution converges to the exponential distribution ${\rm Exp} (\lambda)$; when $\lambda = 1$, \eqref{steinPE} reduces to the Stein equation (4.2) in \cite{pekozandrollin2011}. For this simplified version, the bound in \cite{pekozandrollin2011} is only $\frac12$ of the bound \eqref{bound1}; this discrepancy arises through our use of the triangle inequality for $\theta > 0$. 
\end{remark}

In the following subsections, we compare a {PE} distribution with a {distribution of a maximum of a random number of i.i.d random variables, with a} generalised Poisson-exponential distribution, and with a Poisson-geometric distribution.

\subsection{Approximating the distribution of the maximum of a random number of i.i.d.\,random variables by a PE distribution}\label{sec:approxrandmax}
 
Let $M\in \N$ be independent of ${\bf{X}} = (X_1, X_2, \ldots)$, a sequence of i.i.d.\,random variables, and let $W=W_1(M, {\bf{X}})$ have pdf of the form \eqref{CCR_class} for $\alpha = 1$. Our first comparison result employs Stein's method.

\begin{corollary} \label{prop:general}
Assume that the $X_i's$ have differentiable pdf $p_X$, cdf  $F_X$, and score function $\rho_X$. Let $W = W_1(M, {\bf{X}}) = \max \{X_1, \ldots, X_M\}$. Then 
\begin{eqnarray}\label{eq:peprop3}
    d_{TV}\left(PE(\theta, \lambda), \mathcal{L}(W_1(M, {\bf{X}}))\right) \le \frac{4}{\lambda} \left(\EE\left|\theta \lambda e^{-\lambda W}- \frac{(G_M'' \circ F_X)(W)}{(G_M' \circ F_X)(W)} p_X(W) \right| \right. \nonumber \\
    \left. + \EE \left| - \lambda  - \rho_X (W) \right|\right) . 
\end{eqnarray} 
If $M$ is a zero-truncated Poisson$(\theta_m)$ random variable then \eqref{eq:peprop3} reduces to
 $$ d_{TV}\left(PE(\theta, \lambda), \mathcal{L}(W_1(M, {\bf{X}}))\right) \le \frac{4}{\lambda} \left(\E |\theta\lambda e^{-\lambda W} - \theta_mp_X(W) | + \E | - \lambda  - \rho_X (W) | \right) . $$
\end{corollary}
\begin{proof} We employ Proposition \ref{CCR_comp}. For a zero-truncated Poisson random variable $N$ with parameter $\theta$ so that $G_N''(\cdot) / G_N'(\cdot) = \theta$, and for $Y \sim {\rm Exp}(\lambda)$ with pdf $f_Y(y) = \lambda e^{-\lambda y}$, using \eqref{eq:gencomp} along with \eqref{boundfx}, and by taking $h$ an indicator function so that  $|| \tilde{h} || \le 2 ||h|| \le 2$, gives \eqref{eq:peprop3}. The simplification when $M$ is a zero-truncated Poisson random variable follows from $G_M''(\cdot) / G_M'(\cdot) = \theta_M$.
\end{proof}
\begin{remark}
    As $-\lambda$ is the score function of the exponential distribution ${\rm Exp}(\lambda)$, for $M$ a zero-truncated Poisson $(\theta)$ random variable the bound in Corollary \ref{prop:general} is close to zero if the density and the score function of $X$ are close to that of ${\rm Exp}(\lambda)$.
\end{remark}
The next result is based on the Lindeberg argument.
\begin{corollary} \label{PE_linderberg}
     Let $W_1(N, {\bf{E}}) \sim \text{PE}(\theta, \lambda)$ and let $W_{1}(M, {\bf{X}}) = \max(X_1, \ldots, X_M)$ have a CCR distribution. Then 
     \begin{eqnarray*}
     \lefteqn{d_{BW}(PE(\theta, \lambda), \mathcal{L}(W_{1}(M, {\bf{X}})) \le   \E M \, \E| E_1 - X_1|} \nonumber \\
     && \quad \quad \quad \quad \quad \quad \quad \quad + \sum_{m,n=1}^\infty \PP (M=m, N=n) \left( H_{\max\{m,n\}} - H_{\min\{m,n\}}\right)
 \end{eqnarray*}
where $H_n$ is the $n^{th}$ harmonic number.

 If $M$ is also a zero-truncated Poisson$(\theta)$ random variable, then 
    $$ d_{BW}(PE(\theta, \lambda), \mathcal{L}(W_{1}(M, {\bf{X}})) \le \frac{\theta}{1 - e^{-\theta} } \E | E_1 - X_1|. $$
 \end{corollary}
 \begin{proof} Proposition \ref{prop:lindeberg} gives 
     \begin{eqnarray*}
     \lefteqn{|\E h(W_{1}(N, {\bf{E}})) - \E  h (W_{1}(M, {\bf{X}})) | \le  \sum_{i=1}^\infty  \E | E_i - X_i| \PP (M \ge i) \| h'\|} \nonumber \\
     &+& \sum_{m,n=1}^\infty \PP (M=m, N=n) \left|\E h(W_{1}(n, {\bf{E}})) -
     \E h(W_{1}(m, {\bf{E}}))\right|\| h'\|.
 \end{eqnarray*}
 Since the random variables in ${\bf{X}}$ are i.i.d.\,and ${\bf{E}}$ are i.i.d.\,exponential with parameter $\lambda$ and the expectation of maximum of $n$ exponential random variables is $\sum_{i=1}^n \frac1i = H_n$, the assertion follows.
 \end{proof}
 
\subsection{Approximating the generalised Poisson-exponential distribution} \label{subsec:GPE}

The PE distribution has an increasing or constant failure rate. To model decreasing failure rate as well, Fatima \& Roohi \cite{fatima2015journal} introduced the family of generalised Poisson-exponential (GPE) distributions. The differentiable pdf of a GPE distribution with parameters $\beta, \theta,\lambda > 0 $, denoted by $GPE(\theta,\lambda,\beta)$, is
\begin{equation} \label{pdfgpe}
    p(x| \theta, \lambda, \beta) 
= \frac{\beta \theta \lambda e^{-\lambda x - \theta \beta e^{-\lambda x}}}{(1 - e^{-\theta})^{\beta}} (1-e^{-\theta + \theta e^{-\lambda x}})^{\beta -1} ; \quad \quad x > 0.
\end{equation}
For $\beta = 1$ the density of the GPE distribution simplifies to that of the PE distribution given in \eqref{pdf}. For ${0 < \beta < 1}$ the pdf of the GPE distribution is monotonically decreasing, while for $\beta \ge 1$ it is unimodal positively skewed with skewness depending on the shape parameters $\beta$ and $\theta$. The shape of the hazard function also depends on these two shape parameters. For example, for $\theta = 1$ and $\lambda =2$, the failure rate is decreasing for $ 0 < \beta <1$ and is increasing for $\beta \ge 1$, as in Figure 2 of \cite{fatima2015journal}.

For a data set from \cite{aarset1987} which consists of 50 observations on the time to first failure of devices, \cite{fatima2015journal} showed that the GPE distribution provides a better fit than the PE and some other candidate distributions. However, a GPE distribution is not as easy to manipulate and interpret as a PE distribution. Hence it is a natural question to quantify the sacrifice when approximating a {GPE} distribution with a PE distribution. Here we note that GPE is not from the CCR family and hence Proposition \ref{prop:lindeberg} cannot be applied. Instead we use Stein's method to bound the approximation error. 

For such an approximation to be intuitive, the failure rate of the approximating distribution should be qualitatively similar. Hence we restrict  attention to the case $\beta \ge 1$, for which both the GPE and the PE distributions have an increasing failure rate. As an aside, we note that for $\beta \ge 1$, the limit of the pdf at $0$ is $0$, while for $0 < \beta < 1$ it is undefined; the latter condition leads to a Stein class for the GPE Stein score operator which differs from the Stein class for the PE Stein score operator.

We note here that if $X \sim GPE(\theta, \lambda, \beta)$, then, for $\beta \ge 1$,
\begin{equation} \label{bound_mean_GPE}
    \E X \le \frac{\beta \theta}{\lambda(1 - e^{-\theta})^{\beta}}. 
\end{equation}
The proof of \eqref{bound_mean_GPE} is deferred to the appendix.

The GPE random variable is not of the CCR form \eqref{eq:CCR_rv} but its score function, given in \eqref{scoreGPE} below, can be derived from its pdf \eqref{pdfgpe} with parameters $\beta, \theta, \lambda > 0$ as
\begin{equation} \label{scoreGPE}
    \rho(x) = \frac{p'(x)}{p(x)} = \lambda \theta e^{-\lambda x} \Bigg[\beta +\frac{(\beta-1) e^{-\theta + \theta e^{-\lambda x}}}{1-e^{-\theta + \theta e^{-\lambda x}}}\Bigg] - \lambda,   \quad \quad x > 0.
\end{equation}
In the following Theorem we bound the distance between a GPE with $\beta \ge 1$ and a PE distribution, using their corresponding score Stein operators in \eqref{comparison}.

\bigskip
\bigskip
\begin{theorem} \label{GPEbound_overall}
Let $W \sim PE(\theta_1, \lambda_1)$, $X \sim GPE(\theta_2, \lambda_2, \beta)$ and $h: {\R^+_>0} \rightarrow \R$ be bounded, and let $\tilde h(x) = h(x) - \E h(W)$. Then for $\lambda_1 \le \lambda_2$ and $\beta \ge 1$ it holds that
\begin{equation}\label{GPEbound}
    |\E h(X)-\E h(W)| \le  \|\Tilde{h}\| \left\{ \frac{\lambda_2}{\lambda_1}|\beta-1|  +  \left|\frac{\lambda_2 \theta_2}{\lambda_1 \theta_1}\beta -  1  \right| +  \left( \frac{\lambda_2}{\lambda_1} - 1\right) \left( \lambda_1 \E X + 2\right) \right\}. 
\end{equation}
\end{theorem}
\begin{proof}
Let $p_W$ and $p_X$ denote the pdfs of $W$ and $X$, and $\mathcal{T}_{X}$ denote a score Stein operator for a GPE distribution. To employ \eqref{comparison}, first we check that $f_W$ as in \eqref{solution} for the PE distribution, for $h$ bounded, is in the Stein class of $\mathcal{T}_X$. Invoking Lemma \ref{allboundsforPE}, $f_W$ is bounded. Now $ {\E} [{\mathcal T}_X f_W (X) ]
=  \int_0^\infty \frac{(f_W p_X)' (x)}{p_X(x)} p_X (x) \mathrm{d} x 
= \lim_{x \to \infty } (f_W p_X) (x) - \lim_{x \rightarrow 0} (f_W p_X) (x),
$ and for $\beta \ge 1$ we have that $ {(f_Wp_{X})(x)} \rightarrow 0$ as $x \rightarrow 0$ and as $x \rightarrow \infty$, showing that $f_W$ is in the Stein class for $\mathcal{T}_{X}$. Applying  \eqref{comparison} with the score functions from \eqref{scorePE} and \eqref{scoreGPE} respectively, we have
\begin{eqnarray*}
    \lefteqn{\E h(X)  -\E h(W)} \nonumber \\ 
    &=& \E \bigg[\lambda_2 \theta_2 e^{-\lambda_2 X} \Bigg(\beta +\frac{(\beta-1) e^{-\theta_2 + \theta_2 e^{-\lambda_2 X}}}{1-e^{-\theta_2 + \theta_2 e^{-\lambda_2 X}}}\Bigg) - \lambda_2 - \lambda_1 (\theta_1 e^{-\lambda_1 X}-1)  \bigg]f_{W}(X) \\
    &\le& \E \left[\lambda_2 \theta_2 (\beta-1) e^{(\lambda_1-\lambda_2) X} \frac{ e^{-\theta_2 + \theta_2 e^{-\lambda_2 X}}}{1-e^{-\theta_2 + \theta_2 e^{-\lambda_2 X}}} e^{-\lambda_1 X}  |f_W(X)|\right] \nonumber \\
    &&+ \E\left|\lambda_2\theta_2\beta e^{-\lambda_2 X} - \lambda_1 \theta_1 e^{-\lambda_1 X}\right||f_W(X)| + \E|\lambda_1 - \lambda_2||f_W(X)|.
\end{eqnarray*}
Next, 
\begin{equation}{   |\theta_2 \lambda_2 e^{-\lambda_2 x} - \theta_1 \lambda_1 e^{-\lambda_1 x}|}\\
   \le  |\theta_2 \lambda_2 \beta  -\theta_1 \lambda_1 |e^{(\lambda_1 - \lambda_2) x} e^{-\lambda_1 x} + (\lambda_2-\lambda_1) x \theta_1 \lambda_1 e^{-\lambda_1 x}, \label{eq:useful}
   \end{equation}
   since $|e^{-x} - 1 | \le x$ for all $x \ge 0$. With $e^{(\lambda_1 - \lambda_2) x} \le 1$ for $\lambda_1 \le \lambda_2$, we write
\begin{eqnarray}
   \lefteqn{ |\E h(X)  -\E h(W)|} \label{diff_PE_GPE_c2}\\
   &\le&  \lambda_2 \theta_2 |\beta-1| \E \left|\frac{ e^{-\theta_2 + \theta_2 e^{-\lambda_2 X}}}{1-e^{-\theta_2 + \theta_2 e^{-\lambda_2 X}}} e^{-\lambda_1 X} f_W(X)\right| + ( \lambda_2 - \lambda_1)  \E|f_W(X)| \nonumber \\
  &&   + \lambda_1 \theta_1 \left|\frac{\lambda_2 \theta_2}{\lambda_1 \theta_1}\beta -  1  \right|\E|e^{-\lambda_1 X}f_W(X)| + \lambda_1 \theta_1 ( \lambda_2 - \lambda_1)  \E X |e^{-\lambda_1X}f_W(X)| . \nonumber 
\end{eqnarray}
Now using \eqref{helpful} as well as  $e^{x} - 1 \ge x$ and $1-e^{-x} \le x$ for $x \ge 0$,
we have
\begin{align}
    \left|\frac{e^{-\theta_2 + \theta_2 e^{-\lambda_2 x}}}{1-e^{-\theta_2 + \theta_2 e^{-\lambda_2 x}}}e^{-\lambda_1 x} f_W(x)\right| &\le \frac{\|\Tilde{h}\|}{\lambda_1 \theta_1} \frac{e^{-\theta_2 + \theta_2 e^{-\lambda_2 x}}}{1-e^{-\theta_2 + \theta_2 e^{-\lambda_2 x}}} (1-e^{-\theta_1(1 - e^{-\lambda_1 x})})  \nonumber \\
    & \le  \frac{\|\Tilde{h}\|}{\lambda_1 \theta_2} \label{bound_func_exp}. 
\end{align}
Using \eqref{boundexp}, \eqref{boundfx}, \eqref{bound_mean_GPE} and \eqref{bound_func_exp} in \eqref{diff_PE_GPE_c2} gives the bound \eqref{GPEbound}.
\end{proof}

\begin{remark}  \label{GPEbound_all_equal}
For $\lambda_1=\lambda_2$ and $\theta_1=\theta_2$ in Theorem \ref{GPEbound_overall}, the bound \eqref{GPEbound} can be improved by a factor of 2, using that with $f_{W}$ the solution of the  PE score Stein equation,
\begin{align*}
    |\E h(X)  -\E h(W)| = \E  \left(\lambda \theta |\beta - 1|\frac{e^{-\lambda X}}{1-e^{-\theta + \theta e^{-\lambda X}}}f_{W}(X) \right).
\end{align*}
With  \eqref{helpful} and $\tilde h(x) = h(x) - \E h(W)$ we obtain 
\begin{equation}\label{GPEbound2}
     |\E h(X)-\E h(W)|\le \|\Tilde{h}\| |\beta -1 |.
\end{equation}
This bound solely depends on the parameter $\beta$ and tends to 0 when $\beta\rightarrow 1$,  which is in line with the fact that for $\beta  \rightarrow 1$,  $GPE(\theta, \lambda, \beta)$ converges to  $PE(\theta, \lambda)$, see \cite{fatima2015journal}.
\end{remark}
The next result follows immediately from Theorem \ref{GPEbound_overall} and \eqref{bound_mean_GPE}.
\begin{corollary} \label{boundsPE1}
Let $W_1 \sim PE(\theta_1, \lambda_1)$ and $W_2 \sim PE(\theta_2, \lambda_2)$ with $\lambda_1 \le \lambda_2$. Let $\mathcal{H} = \{h: \R^+_{>0} \rightarrow \mathbb{R}, \|h\| \le 1\}$. Then for all $h \in \mathcal{H}$, letting $\tilde h(w) = h(w) - \E h(W)$, we have from \eqref{GPEbound} with $\beta = 1$ that 
\begin{equation} \label{PEbound} 
   |\E h(W_2)-\E h(W_1)| \le  \|\Tilde{h}\| \left\{ \left|\frac{\lambda_2 \theta_2}{\lambda_1 \theta_1} -  1  \right| +  \left( \frac{\lambda_2}{\lambda_1} - 1\right) \left( \frac{\lambda_1 \theta_2}{\lambda_2 (1 - e^{-\theta_2})}  + 2\right) \right\}.
\end{equation}
\end{corollary}

\begin{remark}
\begin{enumerate}
\item For $\|h\| \le 1$, so that $\|\Tilde{h}\| \le 2$, the bounds can easily be converted into bounds in total variation distance using \eqref{eq:dtv}.
\item To bound $d_{\mathcal{H}}(PE(\theta_1,\lambda_1), GPE(\theta_2,\lambda_2,\beta))$ when $\lambda_1 > \lambda_2$ we can use 
\begin{align*}
    d_{\mathcal{H}} (GPE(\theta_2,\lambda_2,\beta) , PE(\theta_1,\lambda_1))  \le & d_{\mathcal{H}} (GPE(\theta_2,\lambda_2,\beta) ,PE(\theta_2,\lambda_2)) \\
    &+ d_{\mathcal{H}}( PE(\theta_2,\lambda_2), PE(\theta_1,\lambda_1))
\end{align*}
and apply Remark \ref{GPEbound_all_equal} and Corollary \ref{boundsPE1}.
\end{enumerate}
\end{remark}

\subsection{Approximating the Poisson-geometric distribution} \label{subsec:PG}

Next we consider the distribution of $W_G = \max\{T_1, \ldots, T_M\}$, where $T_1, T_2, \ldots \in \{1, 2, \ldots\}$ are i.i.d.\,Geometric($p$) random variables, and $M$ is an independent zero truncated Poisson($\theta$) random variable; the distribution of $W_G$ is the  {\it Poisson-geometric (PG) distribution}; it has pmf
\begin{align}
     \PP(W_G=w) = \frac{e^{-(1-p)^w\theta}- e^{-(1-p)^{w-1}\theta}}{(1-e^{-\theta})}, \quad \quad {w} \in \N, \label{PGdist}
\end{align}
and the discrete backward score function
$$\rho_{W_{G}}(w) = 1- \frac{e^{-(1-p)^{w-1}\theta}- e^{-(1-p)^{w-2}\theta}}{e^{-(1-p)^{w}\theta}- e^{-(1-p)^{w-1}\theta}}, \quad w \in \N.$$  

For $T \sim \text{Geometric}(\lambda/n)$, the distribution of $n^{-1} {T}$ converges to ${\rm Exp}(\lambda)$ in probability, and hence it is plausible to approximate the distribution of $Z_n={W_{G,n}}/n$, for ${W_{G,n}} \sim PG(\theta, \lambda/n)$, by a corresponding Poisson-exponential distribution. With $q_n = 1-\lambda/n $  and $\Tilde{\rho}_n (z) = \rho_{W_G} (nz)$,
$$ \Tilde{\rho}_n (z) = 1- \frac{e^{-q_n^{nz-1}\theta}- e^{-q_n^{nz-2}\theta}}{e^{-q_n^{nz}\theta}- e^{-q_n^{nz-1}\theta}}, \quad nz \in \N.$$ 
This function is the ratio of two exponential functions, 
complicating the comparison using \eqref{comparison}.
To simplify the comparison we use Proposition \ref{prop:comp}. From \eqref{eq:discretestin}, a standardised PG Stein operator for $Z_n$ is 
\begin{equation} \label{eq:trans3}
  {\mathcal{T}_{Z_n} ^{(d)}} (g(z)) =  g(z)  d(z) - g \left( z - \frac1n \right) d \left( z - \frac1n \right)  +\Tilde{\rho}_n(z) g \left( z - \frac1n \right)d \left( z - \frac1n \right);
\end{equation} 
here we choose 
\begin{equation} \label{eq:d} d(z) = e^{-q_n^{nz + 1}\theta}- e^{-q_n^{nz}\theta}
= e^{ -\theta {q_n^{nz}}} \left( e^{ -\theta (q_n-1) q_n^{nz}} -1 \right)\end{equation}
so that 
$ \Tilde{\rho}_n(z) d\left( z - \frac1n \right) = e^{-q_n^{nz }\theta}- 2 e^{-q_n^{nz-1}\theta} + e^{-q_n^{nz -2}\theta}.$
As $nd(z){\rightarrow} \lambda \theta e^{-\lambda z - \theta e^{-\lambda z}}$ as $n \rightarrow \infty$, for the approximating PE distribution we choose 
$$ c(z) =  \frac{\lambda \theta}{n}e^{-\lambda z - \theta e^{-\lambda z}}, \quad \quad for \quad z > 0,$$ 
as a standardisation function in \eqref{stand_stein_cont},  giving rise to the  standardised PE Stein equation,
\begin{equation} \label{mPEStein}
   c(w) g'(w)  + [c(w)  \lambda(\theta e^{-\lambda w}-1)+ c'(w)] g(w) = h(w) - \E h(W),
\end{equation}
for $W \sim PE(\theta, \lambda)$. Again we can bound the solution of this Stein equation as follows.
\begin{lemma} \label{lem:scaledbound}
For $W \sim PE(\theta, \lambda)$ and g(w), the solution of Stein equation \eqref{mPEStein}, for bounded differentiable $h$ such that $\|h\| \le 1$ and $\|h'\| \le 1$ with $\Tilde{h} (x) = h(x) - \E h(W)$, and $ c(w) =  \frac{\lambda \theta}{n}e^{-\lambda {w} - \theta e^{-\lambda {w}}}$, we have
  \begin{eqnarray}
    |e^{-2\lambda w - \theta e^{-\lambda w}} g(w)| &\le& \frac{n}{\lambda^2\theta^2} \|\Tilde{h}\|, \label{msbound1} \\
    \bigg|e^{-\lambda w - \theta e^{-\lambda w}} (\theta e^{-\lambda w} - 1)g(w) \bigg| &\le& \frac{n}{\lambda^2 \theta}\|\Tilde{h}\|, \label{msbound2}\\
    |e^{-\lambda w - \theta e^{-\lambda w}}g(w)| &\le& 2\frac{n}{\lambda^2\theta} \|\Tilde{h}\|, \label{msbound5} \\
    \bigg|e^{-\lambda w - \theta e^{-\lambda w}}g'(w) \bigg| &\le&  3 \frac{n}{\lambda \theta} \|\Tilde{h}\|, \label{msbound3}\\
    \bigg|\frac{\lambda \theta}{n}e^{-\lambda w - \theta e^{-\lambda w}}g''(w)\bigg| &\le&  \|h'\| + 9 \lambda\theta \|\Tilde{h}\| + 11 \lambda \|\Tilde{h}\|. \label{msbound4}
  \end{eqnarray}
and 
  \begin{equation} \label{msbound6}
    \frac{c(x)}{ c\left( x + \frac{\rho}{n}\right)} \le 
    \begin{cases}
    e^{\theta (e^{ \frac{\lambda}{n}} -1)}, & \rho = -1; \\
    e^{\frac{\lambda}{n}}, & 0 < \rho <1. 
    \end{cases}
\end{equation}
\end{lemma}
\begin{proof}
As in \eqref{mPEStein}, $cg = f$  is the solution of the Stein equation \eqref{steinPE} we use the bounds for $f$ to bound $g$. The bound \eqref{boundexp}  in  Lemma  \eqref{allboundsforPE} immediately gives \eqref{msbound1}. Also $c'(w) = c(w) \lambda (\theta e^{-\lambda w} - 1) $ so that $c'(w)g(w) = \lambda (\theta e^{-\lambda w} - 1) f(w)$; using \eqref{bound2} we get \eqref{msbound2}. Combining\eqref{msbound1}, \eqref{msbound2} and the triangle inequality we obtain \eqref{msbound5}.

Since $(cg)' = cg' + c'g$, rearranging and using \eqref{bound1}, \eqref{bound2} with the triangle inequality gives \eqref{msbound3}. For $(cg)'' - c''g - 2c'g' = cg''$, using \eqref{msbound3}, 
$$
    |c'(w)g'(w)| \le \lambda|\theta e^{-\lambda w} - 1|c(w)g'(w)  \le 3\lambda(\theta+1) \|\Tilde{h}\|.
$$
For $c''(w)g(w) = \lambda^2(\theta e^{-\lambda w} -1)^2 c(w)g(w) - \lambda^2\theta e^{-\lambda w}c(w)g(w)$, using the triangle inequality and \eqref{msbound2} and \eqref{msbound1} we get
$$
    |c''(w)g(w)| \le \lambda|\theta e^{-\lambda w} - 1|\|\Tilde{h}\| + \lambda \|\Tilde{h}\| \le \lambda\theta \|\Tilde{h}\| + 2\lambda \|\Tilde{h}\|.
$$
These two results along with \eqref{bound3} give  \eqref{msbound4}. Now for any $0 < \rho < 1$, 
$$ 
 \frac{c(x)}{ c\left( x + \frac{\rho}{n}\right)}
 = e^{\lambda \frac{\rho}{n} - \theta e^{-\lambda x} (1- e^{-\lambda \frac{\rho}{n}})  } \le e^{\lambda \frac{\rho}{n}};
$$
 for $\rho = -1$ we have 
$\frac{c(x)}{ c\left( x - \frac{1}{n}\right)}
 = e^{- \frac{\lambda}{n} + \theta e^{-\lambda x} (e^{ \frac{\lambda}{n}} -1)  } \le e^{ \theta (e^{\frac{\lambda }{n}} -1)}$, yielding \eqref{msbound6}.
\end{proof}

\begin{theorem} \label{PGD}
 Let $W_{G,n} \sim PG(\theta, p_n) $ with $p_n = \lambda/n;  0< \lambda < n$, and let $W \sim PE(\theta,\lambda)$.  Then 
 for the scaled Poisson-geometric random variable $Z_n = W_{G,n}/n$, and for any bounded function $h$ with bounded first derivative, we have
 \begin{align} \label{PGDbound}
     &{|\E h(Z_n) - \E h(W)| \le \frac{e^{\frac{\lambda}{n}}}{n} \left(1+\frac{e^{ \frac{\theta\lambda}{n e}}}{n}B_1(\theta, \lambda)\right)\|h'\|} \nonumber \\
     &+ \frac{1}{n} \left(\left(1-\frac{\lambda}{n}\right)^{-6} e^{\theta e^{ \frac{\lambda}{n}} + \frac{\theta\lambda}{n} + 2\frac{\theta\lambda}{n e} + \frac{\lambda}{n}}B_2(\theta, \lambda) + \frac{1}{2n} e^{\frac{\lambda}{n} + \frac{\theta\lambda}{n e}}B_3(\theta, \lambda)\right)\|\tilde{h}\| ,
 \end{align}
where 
\begin{eqnarray*}
    B_1(\theta, \lambda) &=&  \lambda \max(1,\theta) +  \frac{e^{\lambda \theta + {\frac{\theta}{e}}}}{2\theta\lambda }  + \frac{\theta\lambda}{2e}\frac{(1+e-e^{-\theta})}{(1-e^{-\theta})} ,\\
 B_2(\theta, \lambda) &=& 11 \lambda + 9 \lambda \theta +  6\lambda \max(1,\theta) +  \frac{3e^{\lambda \theta + {\frac{\theta}{e}}}}{\theta \lambda} + 3\theta\lambda \frac{(1+e-e^{-\theta})}{(1-e^{-\theta})e}  \nonumber \\
    &&  +  4\lambda e^{-\theta}\left(1+ 2 \max(1,\theta) + \frac{(3e + 1) \theta}{e}+ \frac{(5e + 3)\theta^2}{3e}  +  \frac{\theta(2\theta +1)}{(1-e^{-\theta})} \right), 
\end{eqnarray*}
and
\begin{eqnarray*}
    B_3(\theta, \lambda) =  (9 \lambda \theta  + 11 \lambda )  \left(2\lambda \max(1,\theta) +  \frac{e^{\lambda \theta + {\frac{\theta}{e}}}}{\theta \lambda} + \theta\lambda \frac{(1+e-e^{-\theta})}{(1-e^{-\theta})e} \right).
\end{eqnarray*}
\end{theorem}
\begin{remark}
\begin{enumerate} 
 \item For fixed $\lambda $ and $\theta$ the bound \eqref{PGDbound} is $O(n^{-1})$. As $n \rightarrow \infty$, for $\lambda = \lambda(n)$ and $\theta = \theta(n)$ the bound decreases to 0 as long as ${\lambda(n) \theta(n)}/{n} \rightarrow 0$ and ${\lambda(n)}/{n} \rightarrow 0$.
 \item 
  Equation \eqref{PGDbound} can be translated into a bound in the bounded Wasserstein distance using ${\rm Lip}_b(1)$ as the class of test functions.
  \end{enumerate}
 \end{remark}
 \begin{proof}
We employ Proposition \ref{prop:comp}. First we note that $ c(w) \rho(w) = c'(w)$ and that 
$ {\tilde \rho}_n (w) d
\left( w - \frac1n \right)
= d \left( w - \frac1n \right) - d \left( w - \frac2n \right)$.  
 Thus, for \eqref{disbound1},  with $h \in {\rm Lip}_b(1)$,  
\begin{align}
   &{|\E h(Z_n) - \E h(W)|} \nonumber \\
    \le &  \Big|
    \E [ n \Delta^{-n} (dg)(Z_n) -  (cg)'(Z_n) ] + \E\Big[ n {\tilde \rho}_n (Z_n) (dg)\left( Z_n - \frac1n \right) - (cg)\left( Z_n\right) \rho(Z_n)\Big] \Big| \nonumber \\
    \le& \E\left|
      n \left\{  g(Z_n)  - g \left( Z_n - \frac1n \right)\right\} c(Z_n) - g'(Z_n) c(Z_n) \right| \label{eq:new1}\\
    & + \E \left|
     n \left\{  g(Z_n)  - g \left( Z_n - \frac1n \right)\right\} \left(d(Z_n) - c(Z_n) -\frac{2}{n} c'(Z_n)\right) \right|\label{eq:new2}\\
    & +
    \E \left| d(Z_n) - d \left( Z_n - \frac2n \right)
    - \frac{2}{n}  c'(Z_n)\right| \left|ng\left( Z_n - \frac1n \right)\right| \label{eq:new3}. 
\end{align}
To bound the term \eqref{eq:new1}, for some $0 < \rho < 1$, we write
\begin{align*}
    \E\left|
     n \left\{  g(Z_n)  - g \left( Z_n - \frac1n \right)\right\} c(Z_n) - g'(Z_n) c(Z_n) \right| 
   & \le \frac1n \E \left|c(Z_n) g''\left( Z_n + \frac{\rho}{n}\right) \right|
   \\
   &\le \frac1n || c g''|| 
   \E \left| \frac{c(Z_n)}{ c\left( Z_n + \frac{\rho}{n}\right)} \right|.
\end{align*}
Then \eqref{msbound4} and \eqref{msbound6} gives the bound
\begin{align} \label{finbound_1}
    \E\left|
     n \left\{  g(Z_n)  - g \left( Z_n - n^{-1} 
     \right)\right\} c(Z_n) - g'(Z_n) c(Z_n) \right| \le \frac{e^{\frac{\lambda}{n}}}{n} \{ || h'|| + 9 \lambda \theta || {\tilde h}|| + 11 \lambda || {\tilde h}||\}. \nonumber \\ 
\end{align}

To bound \eqref{eq:new2}, we let $\tau(z) = e^{\lambda z + \theta e^{-\lambda z}}$, so that $\tau^{-1} g = \frac{n}{\theta\lambda}cg$ can be bounded as in Lemma \ref{lem:scaledbound}, and write 
\begin{eqnarray*}
    \lefteqn{\E \left|
      n \left\{  g(Z_n)  - g \left( Z_n - \frac1n \right)\right\} \left(d(Z_n) - c(Z_n) -\frac{2}{n} c'(Z_n)\right) \right|} \\
    && \quad \quad \quad \quad \quad \quad \quad = \E \left|
     l(Z_n)  n \left\{  g(Z_n)  - g \left( Z_n - \frac1n \right)\right\}\tau^{-1}(Z_n) \right|,
\end{eqnarray*}
where $l(Z_n) = \tau(Z_n)\left(d(Z_n) - c(Z_n) -\frac{2}{n} c'(Z_n)\right)$. We show in Appendix \ref{sec:proof} that
\begin{align} \label{eq:boundcd}
    |l(Z_n)| \le \frac{2\lambda^2 \theta}{n^2} \max(1,\theta) + \frac{1}{n^2} e^{\lambda \theta + {\frac{\theta}{e}}} + \frac{\theta^2\lambda^2}{n^2 } e^{ \frac{\theta\lambda}{n e}-1} + \frac{\theta\lambda^3Z_n}{n^2} e^{ \frac{\theta\lambda}{n e}}. 
\end{align}
 By Taylor expansion, for some $0 < \epsilon < 1$
\begin{eqnarray}
    n \left\{  g(Z_n)  - g \left( Z_n - \frac1n \right)\right\}\tau^{-1}(Z_n) &=  g'(Z_n)\tau^{-1}(Z_n)  + \frac{1}{2n}g''\left(Z_n + \frac{\epsilon}{n}\right)\tau^{-1}(Z_n)  \nonumber \\
    &= \frac{n}{\theta\lambda}\|cg'\|  + \frac{1}{2\theta\lambda} \|cg''\| \frac{c(Z_n)}{ c\left(Z_n + \frac{\epsilon}{n}\right)}. \label{eq:taylg}
\end{eqnarray}
Using \eqref{eq:boundcd} and \eqref{eq:taylg} along with \eqref{msbound3}, \eqref{msbound4} and \eqref{msbound6} yields
\begin{eqnarray}
   \lefteqn{\E \left|
     n \left\{  g(Z_n)  - g \left( Z_n - \frac1n \right)\right\} \left(d(Z_n) - c(Z_n) -\frac{2}{n}c'(Z_n)\right) \right|} \nonumber \\
     &\le&  \left(\frac{2\lambda^2 \theta}{n^2} \max(1,\theta) + \frac{1}{n^2} e^{\lambda \theta + {\frac{\theta}{e}}}+ \frac{\theta^2\lambda^2}{n^2 } e^{ \frac{\theta\lambda}{n e}-1} + \frac{\theta\lambda^3 \E Z_n}{n^2} e^{ \frac{\theta\lambda}{n e}}\right) \nonumber \\
     &&\left\{  \frac{n}{\theta\lambda}3\|\tilde{h}\|  + \frac{1}{2\theta\lambda} \{ || h'|| + 9 \lambda \theta || {\tilde h}|| + 11 \lambda || {\tilde h}||\} e^{\frac{\lambda}{n}} \right\} .
    \label{finbound_2}
\end{eqnarray}
Finally, to bound \eqref{eq:new3}, we show in Appendix \ref{sec:proof} that
\begin{eqnarray}\label{appendix_A}
    \lefteqn{\left|\frac{2}{n}  c'(z) - d(z) + d \left( z - \frac2n \right)\right|\left|ng\left( z - \frac1n\right)\right| } \\
    &\le& 2\lambda \left\{ \left(1-\frac{\lambda}{n}\right)^{-2} -1 \right\} \max(1,\theta) c(z)  \left| g\left( z - \frac1n\right)\right| \nonumber \\
    && + \frac{\theta\lambda^2}{n} \left(1-\frac{\lambda}{n}\right)^{-2}  e^{{\frac{\theta\lambda}{n e}}} \bigg\{ \frac{1}{3} \left(1-\frac{\lambda}{n}\right)^{-1} \nonumber  \left[\theta + \theta e^{\frac{\theta\lambda}{n}} + 6 e^{\frac{\theta\lambda}{n}} + 12 \left(1-\frac{\lambda}{n}\right)^{-1} 
    \right. 
    \\
    & & \left. \quad \quad \quad \quad \quad \quad \quad \quad \quad \quad + 8\theta \left(1-\frac{\lambda}{n}\right)^{-3}\right]  + 2\left({\frac{\theta}{e}}  + 2\lambda z \right)\bigg\}\nonumber  c(z) \Big| g\left( z - \frac1n\right)\Big| \\&& + \frac{2\lambda^2}{n}\Big( 1- \frac{\lambda}{n}\Big)^{-2}e^{{\frac{\theta\lambda}{n e}}}\Big[ \Big(1-\frac{\lambda}{n}\Big)^{-1}e^\frac{\theta \lambda}{n}  + \lambda z + {\frac{\theta}{e}} \Big] c(z) \left|g\left( z - \frac1n\right)\right|. 
    \nonumber 
\end{eqnarray}
 Note that
\begin{equation} \label{simp_1}
   0 \le \left(1-\frac{\lambda}{n}\right)^{-2} -1
   = \frac{\lambda}{n} \left(1-\frac{\lambda}{n}\right)^{-2}\left( 2 - \frac{\lambda}{n}  \right),
\end{equation}
and, using \eqref{msbound5} and \eqref{msbound6}, that
$$c(z) \left|g\left( z - \frac1n\right)\right|= \left|c(z) g\left( z - \frac1n\right)\right| \le \|cg\|\left| \frac{c(z)}{ c\left( z - \frac{1}{n}\right)} \right| \le \frac{2}{\lambda} e^{\theta (e^{ \frac{\lambda}{n}} -1)}\|\Tilde{h}\|.$$ 
Combining this with \eqref{appendix_A} and \eqref{simp_1} we bound \eqref{eq:new3} as
\begin{eqnarray} \label{finbound_3}
    \lefteqn{\E \left| d(Z_n) - d \left( Z_n - \frac2n \right)
    - \frac{2}{n}  c'(Z_n)\right| \left|ng\left( Z_n - \frac1n \right)\right|}  \\
    & \le & \frac{4 \lambda}{n}\max(1,\theta) \left(1-\frac{\lambda}{n}\right)^{-2}\left|\frac{\lambda}{n} - 2 \right|  e^{\theta (e^{ \frac{\lambda}{n}} -1)}\|\Tilde{h}\| + \frac{2\theta\lambda}{n} \left(1-\frac{\lambda}{n}\right)^{-2}  e^{{\frac{\theta\lambda}{n e}}} 
    e^{\theta (e^{ \frac{\lambda}{n}} -1)}\|\Tilde{h}\| \nonumber \\&& \bigg\{ \frac{1}{3} \left(1-\frac{\lambda}{n}\right)^{-1}  \left[\theta + \theta e^{\frac{\theta\lambda}{n}} + 6 e^{\frac{\theta\lambda}{n}} + 12 \left(1-\frac{\lambda}{n}\right)^{-1} + 8\theta \left(1-\frac{\lambda}{n}\right)^{-3}\right]  \nonumber \\
    &&+ 2\left({\frac{\theta}{e}}  + 2\lambda \E Z_n \right) + \frac{2}{\theta}
    \bigg[ \left(1-\frac{\lambda}{n}\right)^{-1}e^{\frac{\theta \lambda}{n}}  + \lambda \E Z_n + {\frac{\theta}{e}} \bigg]  \bigg\}. \nonumber  
\end{eqnarray}

To bound  $\E Z_n $, we argue as for \eqref{bound_mean_GPE}; $W_{G,n}  = \max \{ T_1, \ldots, T_M\} \le \sum_{i=1}^{M} T_i$ and so $  \E {W_{G,n}} \le \E   \sum_{i=1}^{M} T_i = \frac{n\theta}{\lambda(1-e^{-\theta})}$, giving  that 
 $  \E Z_n  = \frac{1}{n} \E W_{G,n}\le \frac{\theta}{\lambda(1-e^{-\theta})}.
$
Adding \eqref{finbound_1}, \eqref{finbound_2} and \eqref{finbound_3} and simplifying gives \eqref{PGDbound}.
\end{proof}

The next result uses instead Proposition \ref{prop:lindeberg} to bound the distance between a PG and a PE distribution.

\begin{corollary}\label{rem:lindeberg}
Let $W = W_1(M,{\bf{E}}) \sim PE(\theta, \lambda)$ and let ${W_{G,n}} = W_1(M,{\bf{G}}) \sim PG(\theta, \frac{\lambda}{n})$ with $G_i \sim \text{Geometric}(\lambda/n)$, be two CCR random variables, and let $Z_n = \frac{{W_{G,n}}}{n}$. Then, for all  bounded  Lipschitz functions $h: \R_{>0}^{+} \rightarrow \R$ 
\begin{equation} \label{bound_L_type} 
     \left| \EE h(W) - \EE h \left(Z_n \right) \right|  \le
     \frac{\theta}{1 - e^{-\theta}} \frac{2}{n}\left(\frac{n - e^{\frac{\lambda}{n}}(n-\lambda)}{\lambda(e^{\frac{\lambda}{n}}-1)}\right)\|h'\| .
\end{equation}
\end{corollary}
\begin{proof}
Using \eqref{prop2inequality} in Proposition \ref{prop:lindeberg} gives  
\begin{align*}
  \left| \EE h(W) - \EE h \left(Z_n \right) \right| = \left| \EE h(W) - \EE h \left(\frac{W_{G,n}}{n} \right) \right|  &\le \frac{\theta}{1 - e^{-\theta} } \E \left|E - \frac{G}{n}\right| || h'||.
\end{align*}
With the coupling $\Tilde{G} = \lceil n{E}\rceil \sim \text{Geometric}(1-e^{-\lambda/n})$, $\tilde{G}$ is stochastically greater or equal to $G$ and $\E |n{E} - G| \le \E |n{E}-\Tilde{G}| + \E (\Tilde{G} - G)$. Moreover,
\begin{align*}
    \E |nE - \Tilde{G}| = \sum_{k=0}^{\infty} \int^{\frac{k}{n}}_{\frac{(k-1)}{n}} \lambda e^{-\lambda x} \left(k - nx \right) dx = \frac{n - e^{\frac{\lambda}{n}}(n-\lambda)}{\lambda(e^{\frac{\lambda}{n}}-1)},
\end{align*}
and 
$ 
\E(\Tilde{G} - G)  = ({1-e^{-\frac{\lambda}{n}}})^{-1} - {n}/{\lambda}.$
Hence the bound \eqref{bound_L_type} follows.
\end{proof}

\begin{remark}
To compare \eqref{PGDbound} and \eqref{bound_L_type}, first note that as no fixed bound on $\| h\| $ is assumed, through re-scaling $h$ we can make $\|\Tilde{h}\|$ as small as desired. Therefore, for a comparison we focus on the terms involving $\|h'\|$. For continuous $\|h'\|$, \eqref{PGDbound} outperforms the bound \eqref{bound_L_type} for large $n$, since  
$ 
\underset{n \rightarrow \infty}{\lim} e^{\frac{\lambda}{n}} \left(1+\frac{e^{ \frac{\theta\lambda}{n e}}}{n}B_1(\theta, \lambda)\right) = 1$ and $\underset{n \rightarrow \infty}{\lim} \frac{\theta}{1 - e^{-\theta}} 2\left(\frac{n - e^{\frac{\lambda}{n}}(n-\lambda)}{\lambda(e^{\frac{\lambda}{n}}-1)}\right) = \frac{\theta}{(1-e^{-\theta})}.$ In particular, for any $n \ge n_0$  the r.h.s.\,of 
\eqref{bound_L_type} is larger than the coefficient of $\|h'\|$ in  the bound \eqref{PGDbound}, where $n_0 =n_0(\theta, \lambda)$  solves $ e^{\frac{\lambda}{n}}  \le \frac{2n\theta}{\lambda(1 - e^{-\theta})(e^{\frac{\lambda}{n}}-1)} \left(\frac{n - e^{\frac{\lambda}{n}}(n-\lambda)}{n+e^{ \frac{\theta\lambda}{n e}}B_1(\theta, \lambda)}\right).$  
Table \ref{tab:outperform} shows such values of $n_0$.

 \begin{table} [H]
        \centering
\caption{Values of $n$ above which the coefficient of $\|h'\|$  is smaller in \eqref{PGDbound}  than  in \eqref{bound_L_type}}
\label{tab:outperform}
        \begin{tabular}{|c|ccccccccc|llll|}
        \hline
         $\lambda$& 1 & 1 & 2 & 2 & 2 & 3 & 3 & 3  & 5   \\
         \hline
         $\theta$ & 1 & 2 & 1 & 2 & 3 & 2 & 3 & 5 & 3   \\
         \hline
            $n_0$ & 9 & 7 & 14 & 16 & 35 & 49 & 374 & 67,775 & 87,731  \\
         \hline
   \end{tabular}
    \end{table}
\end{remark}

\section{Application to the exponential-geometric distribution} \label{sec:EG} 

The {\it exponential geometric distribution} EG($\lambda, p$) introduced by \cite{ADAMIDIS199835} is the distribution of the minimum of $N$ i.i.d.\,random variables
${\bf E}= (E_1, E_2,\ldots)$ with $E_i \sim {\rm Exp}(\lambda), i \in \mathbb{N},$ and $N$ is a Geometric($p$) random variable, and independent of all $E_i$'s. We set $q=1-p$. An exponential geometric  random variable  $W = W_{-1}(N, {\bf{E}})$ is thus a CCR random variable. As $G_N(u) = \frac{up}{(1-qu)}$, $G'_N(u) = \frac{p}{(1-qu)^2}$ and $G''_N(u) = \frac{2pq}{(1-qu)^3}$, the pdf of $W \sim \text{EG}(\lambda, p)$ is
\begin{equation} \label{eg_pdf}
    f(w; \lambda, p) = \frac{\lambda p e^{-\lambda w}} {(1-qe^{-\lambda w})^2}, \quad w, \lambda
    >0, 0 \le p \le 1,  
\end{equation}
with score function
\begin{equation} \label{EG_score}
    \rho(w) = - \lambda \left(\frac{1+qe^{-\lambda w}}{1-qe^{-\lambda w}}\right).
\end{equation}
This gives the score Stein operator
\begin{equation}
    \T g(w) = g'(w) - \lambda \left(\frac{1+qe^{-\lambda w}}{1-qe^{-\lambda w}}\right)g(w),
\end{equation}
and the Stein equation
\begin{equation} \label{steinEG}
    g'(w) - \lambda \left(\frac{1+qe^{-\lambda w}}{1-qe^{-\lambda w}}\right)g(w) = h(w) - \EE h(W).
\end{equation}
The next Lemma bounds the solution \eqref{solution} of this EG Stein equation.

\begin{lemma} \label{allboundsforEG} 
For any bounded test function $h:  \R^+_{>0} \rightarrow \R$, we let $\Tilde{h}(w) = h(w)-\E h(W)$ for  $W \sim EG(\lambda,p)$. Then its solution $g=g_h$ of the EG Stein equation \eqref{steinEG} satisfies 
\begin{eqnarray}
    |g(w)| &\le& \frac{\|\tilde{h}\|}{\lambda},  \label{boundgx} \\
    |g'(w)| &\le&  \frac{2\|\tilde{h}\|}{p}. \label{bound_dgx}
\end{eqnarray}

\end{lemma}
\begin{proof}
    For the solution \eqref{solution} of the EG Stein equation \eqref{steinEG} we have 
    \begin{eqnarray*}
    |g(w)| \le \|\tilde{h}\| \left| \frac{ -1 }{f(w)} \int_w^{\infty} f(t)\mathrm{d}t \right|
    = \|\tilde{h}\| \left| \frac{ 1-F(w)}{f(w)} \right|
     = \|\tilde{h}\| \left| \frac{(1-qe^{-\lambda w})}{\lambda} \right| \le  \frac{\|\tilde{h}\|}{\lambda}.
\end{eqnarray*}
The last inequality follows since with $0 \le e^{-\lambda w} \le 1$ and $0 \le q \le 1$, we have $1 -q \le 1 - qe^{-\lambda w} \le 1$, and $1 \le 1 + qe^{-\lambda w} \le 1+q$. For a bound on  $| g'(w)| $ we use the Stein equation \eqref{steinEG} and the triangle inequality to obtain 
\begin{eqnarray*}
     |g'(w)| \le \|\tilde{h}\| + \left| \lambda \left(\frac{1+qe^{-\lambda w}}{1-qe^{-\lambda w}}\right)\right| \frac{\|\tilde{h}\|}{\lambda}
      \le \|\tilde{h}\| \left(1 + \frac{1+q}{1-q}\right).
\end{eqnarray*}
Simplifying the last inequality gives \eqref{bound_dgx}.
\end{proof}

\subsection{The minimum of a geometric number of i.i.d.\,random variables} \label{sec:approxrandmin}
Let $N\in \N$ a random variable which is independent of ${\bf{X}} = (X_1, X_2, \ldots)$, a sequence of i.i.d.\,random variables, and let $W=W_{-1}(N, {\bf{X}})$ have pdf of the form \eqref{CCR_class} for $\alpha = -1$. In this section we approximate its distribution by an EG distribution. First we employ Proposition \ref{prop:general}.

 \begin{corollary} \label{prop:eg_min}
 Assume that the $X_i's$ have cdf  $F_X$, differentiable pdf $p_X$,  and score function $\rho_X$ and let $N \sim \text{Geometric}(p)$. Then
\begin{eqnarray*}
     \lefteqn{d_{TV}(\text{EG}(\lambda,p),\mathcal{L}(W))\le \frac{2}{\lambda}\left| \EE ( -\lambda- \rho_X(W))\right.} \\
      && \left. \quad \quad \quad \quad \quad \quad \quad \quad \quad - \EE \left( \frac{2(1-p)\lambda e^{-\lambda W}}{1-(1-p)e^{-\lambda W}} -  \frac{2(1-p)f_X(W)}{1-(1-p)(1-F_X(W))} \right)  \right|.
\end{eqnarray*}
\end{corollary}
\begin{proof}
Substituting $\frac{G_N''(1-F(w))}{G_N'(1-F(w))} = \frac{2(1-p)f(w)}{1-(1-p)(1-F(w))}$ 
and the score function of the exponential distribution in \eqref{eq:gencomp}, taking $h$ an indicator function so that  $|| \tilde{h} || \le 2 ||h|| \le 2$ and using \eqref{boundgx} gives the bound.
\end{proof}

Next, we instead use Proposition \ref{prop:lindeberg}; the corollary follows immediately from \eqref{prop2inequality}.
\begin{corollary}
For $W = W_{-1}(N, {\bf{X}})$ with $N \sim \text{Geometric}(p)$,
$$ d_{BW}(EG(\lambda,p), \mathcal{L}(W) \le {\E | E - X|}/{p} . $$
\end{corollary}

\subsection{Approximating the extended
exponential-geometric distribution} \label{EGtoEEG}

Motivated by population heterogeneity, Adamidis et al. \cite{ADAMIDIS2005259} developed the extended exponential-geometric distribution (EEG) by assuming that individual units in a population have increasing failure rates that depend on a random scale parameter $A$. Their lifetimes $X/A$  are modelled by a modified extreme value distribution; if $A = \alpha$ then the pdf is $f(x|\alpha;\beta) = \alpha \beta e^{\beta x + \alpha(1-e^{\beta x})}$, where $x, \alpha, \beta \in \R^+_{>0}$; it is assumed that $A$ has ${\rm Exp}(\gamma)$ distribution. Then the unconditional lifetime distribution  $X$ has pdf 
\begin{equation} \label{pdf_eeg}
    f(x;\beta, \gamma) =   \frac{\beta \gamma e^{-\beta x}} {(1- (1-\gamma) e^{-\beta x})^2}, \quad x>0;
\end{equation}
with $\beta, \gamma \in \R^+_{>0}$;  we use the notation $X \sim EEG(\beta, \gamma)$. Its 
 score function is 
\begin{equation} \label{EEG_score}
    \rho(x) = \frac{f'(x)}{f(x)} = - \beta \left(\frac{1+(1-\gamma)e^{-\beta x}}{1-(1-\gamma)e^{-\beta x}}\right).
\end{equation}
This distribution is not in the CCR family. However, the exponential-geometric distribution $EG(\beta, \gamma)$ is a special case when $\gamma \in (0,1)$. To assess their total variation distance, we use the general methodology developed in Section \ref{sec:gen_Stein}.

\bigskip
\begin{theorem} \label{theo:EG_EEG}
    For $X \sim EEG(\beta, \gamma)$ and $W \sim EG(\lambda, p)$ with $\beta, \gamma, \lambda \in \R^+_{>0}$, $p \in (0,1)$ and a bounded test function $h$, we have
    \begin{align} \label{bound_EG_EEG}
    |\E h(X)-\E h(W)| \le \begin{cases}
        \frac{\|\tilde{h}\|}{\lambda}\left[| \lambda - \beta| \left( \frac{2-p}{p} + \frac{2(1-\gamma)}{e  \gamma^2} \right) + \frac{2\beta} {\min(\gamma, p)^2} | p - \gamma|  \right] &, 0 < \gamma <1 ;\\
        \frac{\|\tilde{h}\|}{\lambda p } \left[\left|\lambda - \beta\right|\left(1+(\gamma-1)(1-p)\right) + (\lambda+\beta)\left(\gamma -p \right)\right] & , \gamma \ge 1.
    \end{cases}
\end{align}
\end{theorem} 
\begin{proof}
Using the score functions \eqref{EG_score} and \eqref{EEG_score} in \eqref{comparison} yields
\begin{eqnarray}
    \lefteqn{|\E h(X)-\E h(W)| = |\E(\rho_X(X)-\rho_Y(X))g(X)|} \nonumber \\
    &=&  \left|\E \left(- \beta \left(\frac{1+(1-\gamma)e^{-\beta X}}{1-(1-\gamma)e^{-\beta X}}\right)+ \lambda \left(\frac{1+ (1-p)e^{-\lambda X}}{1- (1-p)e^{-\lambda X}}\right)\right)g(X)\right| \nonumber \\
   &\le& \frac{\|\tilde{h}\|}{\lambda} \E \left| - \beta \left(\frac{1+(1-\gamma)e^{-\beta X}}{1-(1-\gamma)e^{-\beta X}}\right)+ \lambda \left(\frac{1+ (1-p)e^{-\lambda X}}{1- (1-p)e^{-\lambda X}}\right)\right|. \label{diff_EG_EEG}
\end{eqnarray}
To bound the expectation in the above equation, we let 
$$R(x) = - \beta \left(\frac{1+(1-\gamma)e^{-\beta x}}{1-(1-\gamma)e^{-\beta x}}\right)+ \lambda \left(\frac{1+ (1-p)e^{-\lambda x}}{1- (1-p)e^{-\lambda x}}\right)$$

\textbf{Case 1:}  $0 < \gamma <1$. In this case we decompose $|R(x)|$ as
\begin{eqnarray}
     \lefteqn{|R(x)| \le  \left|(\lambda - \beta) 
   \frac{1+ (1-p)e^{-\lambda x}}{1- (1-p)e^{-\lambda x}}\right|} 
   \label{eq:term1}
   \\
   &+&  \left|\beta \left( \frac{1+ (1-p)e^{-\lambda x}}{1- (1-p)e^{-\lambda x}} - \frac{1+ (1-\gamma)e^{-\lambda x}}{1- (1-\gamma)e^{-\lambda x}}\right) \right| \label{eq:term2}
   \\
   &+&  \left|\beta \left( \frac{1+ (1-\gamma)e^{-\lambda x}}{1- (1-\gamma)e^{-\lambda x}} - \frac{1+ (1-\gamma)e^{-\beta x}}{1- (1-\gamma)e^{-\beta x}}\right) \right|, \label{eq:term3}
\end{eqnarray}
and bound these terms separately.
For \eqref{eq:term1} we use that 
$ \frac{1 + \alpha q}{1 - \alpha q} \le \frac{1+q}{1-q}$ when $\alpha \in [0,1]$ and $q \in (0,1).$ Thus
\begin{eqnarray*}
    \left\vert (\lambda - \beta) 
   \frac{1+ (1-p)e^{-\lambda x}}{1- (1-p)e^{-\lambda x}}\right\vert
   &\le& | \lambda - \beta| \frac{2-p}{p}. 
\end{eqnarray*}

For \eqref{eq:term2}, Taylor expansion around $1-\gamma$ of the function 
$f(q) = \frac{1+aq }{1-aq} $ with $a = e^{-\lambda x} \in (0,1)$, and $q=1-p$, gives
$f'(q) = \frac{a}{1 - aq} + \frac{a(1+aq)}{(1-aq)^2} = \frac{2a}{(1-aq)^2}>0$. Moreover, 
for $0< a< 1$ we have 
$  \frac{2a}{(1-aq)^2} \le \frac{2}{(1-q)^2}$ and thus,  for $\theta \in (0,1),$ we have 
$0< f'(\theta (1-p) + (1- \theta) (1-\gamma)) \le \frac{2}{(1- \max(1-\gamma, 1-p))^2} = \frac{2}{(\min (\gamma, p))^2}$.
Hence, 
\begin{eqnarray*}
   \left\vert  \beta \left( \frac{1+ (1-p)e^{-\lambda x}}{1- (1-p)e^{-\lambda x}} - \frac{1+ (1-\gamma)e^{-\lambda x}}{1- (1-\gamma)e^{-\lambda x}}\right)\right\vert &\le& 
   \beta  \frac{2| p - \gamma|}{(\min (\gamma, p))^2}.
\end{eqnarray*}
For \eqref{eq:term3}, first order Taylor expansion of the function
$f(\beta) = \frac{1+(1-\gamma) e^{-\beta x}}{1-(1-\gamma) e^{-\beta x}} $ gives 
\begin{eqnarray*}
\left\vert \beta \left( \frac{1+ (1-\gamma)e^{-\lambda x}}{1- (1-\gamma)e^{-\lambda x}} - \frac{1+ (1-\gamma)e^{-\beta x}}{1- (1-\gamma)e^{-\beta x}}\right) \right\vert 
&\le& \beta | \lambda - \beta| | f'(\theta \lambda + (1-\theta) \beta)| 
\end{eqnarray*}
for some $\theta \in (0,1).$
Now the function $f'(\beta) = \frac{{2 x (1-\gamma)} e^{-\beta x}}{(1 - (1-\gamma) e^{-\beta x})^2}$ is positive; moreover $x e^{-\beta x} \le (e \beta)^{-1}$ for $x \ge 0$, so that
$f'(\beta) \le \frac{1}{\beta e} \frac{2(1-\gamma)}{ (1 - (1-\gamma) e^{-\beta x})^2} \le \frac{2(1-\gamma)}{e \beta \gamma^2}.$ 
Hence we can bound 
\begin{eqnarray*}
\left\vert \beta \left( \frac{1+ (1-\gamma)e^{-\lambda x}}{1- (1-\gamma)e^{-\lambda x}} - \frac{1+ (1-\gamma)e^{-\beta x}}{1- (1-\gamma)e^{-\beta x}}\right) \right\vert 
&\le&  \frac{2(1-\gamma)}{e  \gamma^2} | \lambda - \beta|
.  
\end{eqnarray*}
This gives as overall bound
\begin{align*}
   |  R(x) | &\le | \lambda - \beta| \left( \frac{2-p}{p} + \frac{2(1-\gamma)}{e  \gamma^2} \right) + \frac{2\beta} {\min(\gamma, p)^2} | p - \gamma| .
\end{align*}
Substituting this bound in \eqref{diff_EG_EEG} gives the final bound for the case when $0 < \gamma <1$.

\textbf{Case 2:} For $\gamma \ge 1$, we have 
\begin{eqnarray*}
    \lefteqn{|R(x)| \le \left|\frac{(\lambda - \beta)(1-(1-\gamma)(1-p)e^{-(\lambda+\beta)x})}{\left(1-(1-\gamma)e^{-\beta x}\right) \left(1- (1-p)e^{-\lambda x}\right)}\right|} 
    \\ 
    &&+ \left|\frac{(\lambda+\beta)\left((1-p)e^{-\lambda x} - (1-\gamma) e^{-\beta x}\right)}{\left(1-(1-\gamma)e^{-\beta x}\right) \left(1- (1-p)e^{-\lambda x}\right)}\right|
    \\
    &\le& \frac{\left|\lambda - \beta\right|\left(1+(\gamma-1)(1-p)\right)}{p} + \frac{(\lambda+\beta)\left(\gamma -p \right)}{p}.
\end{eqnarray*}
Substituting this result in \eqref{diff_EG_EEG} gives the bound for $\gamma \ge 1$ in \eqref{bound_EG_EEG}.
\end{proof}

\begin{remark}
\begin{enumerate}
    \item The bounds are not optimised for numerical value,  but for $\beta = \lambda$ and $\gamma = p$, \eqref{bound_EG_EEG} reduces to zero, as it should.
    \item For $\gamma \in (0,1)$, \eqref{bound_EG_EEG} gives a bound on the distance between an EG$(\beta, \gamma)$ and an EG$(\lambda, p)$ distribution.
\end{enumerate}
\end{remark}

\section{The maximum waiting time of sequence patterns in Bernoulli trials} \label{sec:patterns}

This application is motivated by the results in Section 4 of \cite{pekoz1996stein} which gives bounds on the distribution of the number of trials preceding the first appearance of a pattern in dependent Bernoulli trials. Here we are interested in the distribution of the maximum of such random variables. 

Consider $M$ independent parallel systems $(X_1^{(i)},X_2^{(i)},...)$, $i=1,2,\ldots ,M$ of possibly dependent Bernoulli$(a)$ trials, which are jointly independent of $M \in \N $. For each sequence $i$, let $I_j^{(i)}$ be the indicator function that a fixed non-overlapping binary sequence pattern of length $k$ occurs starting at $X_j^{(i)}$; the pattern may be specific to sequence $i$. Let $V_i =\min\{j: I_j^{(i)}=1\}$ denote the first occurrence of the pattern of interest in the $i^{th}$ system; assume that $\PP(V_i = 1) = p$ for all $i \in \N $.  We denote the maximum waiting time for the occurrence of a corresponding sequence pattern in all $M$ parallel systems by $W = W_1(M,{\bf{V}})$ where ${\bf{V}} = (V_1, V_2 \ldots)$.

\begin{example}
If the Bernoulli trials are independent, and the pattern of interest is a run of 1's of length $k$, starting with a $0$ and followed by $k$ 1's, then $\PP(V_i = 1) = p = (1-a)a^k$, as given in Corollary 2 of \cite{pekoz1996stein}.
Intuitively, the waiting time for the occurrence of this pattern in an individual sequence is approximately geometric with parameter $p$.
For an approximation by a Poisson-exponential distribution 
we are particularly interested in instances where we can write the probability $\PP(V_i = 1) = p$ as $p = \lambda/n$; here this leads to scaling the run length $k$ as $k \sim a \log n$.

\end{example}

\begin{corollary} \label{cor:seq_patt}
In the above setting, let $U_n = W/n$, with $p =\frac{\lambda}{n}$, and $W_1(M', {\bf{E}}) \sim PE(\theta, \lambda)$, with $M'$, a zero-truncated Poisson random variable and $E_i \sim {\rm Exp}(\lambda)$, we have
\begin{eqnarray*} 
    \lefteqn{d_{BW}(\mathcal{L}(U_n), PE(\theta, \lambda))} \\ 
    &\le&  \frac{2\lambda(k-1)}{n}\EE M + \frac{1}{\ln(p^{-1})} \E \sum_{k= \min(M,M')}^{\max(M,M')}\frac1k
    + \frac{e^{\frac{\lambda}{n}}}{n} \left(1+\frac{e^{ \frac{\theta\lambda}{n e}}}{n}B_1(\theta, \lambda)\right) \nonumber \\
     &+& \frac{1}{n} \left(\left(1-\frac{\lambda}{n}\right)^{-6} e^{\theta e^{ \frac{\lambda}{n}} + 2\frac{\theta\lambda}{n} + \frac{\theta\lambda}{n e} + \frac{\lambda}{n}}B_2(\theta, \lambda) + \frac{1}{2n} e^{\frac{\lambda}{n} + \frac{\theta\lambda}{n e}}B_3(\theta, \lambda)\right),
\end{eqnarray*}
where $B_1(\theta, \lambda)$, $B_2(\theta, \lambda)$, and $B_3(\theta, \lambda)$ are as given in Theorem \ref{PGD}.
\end{corollary}

\begin{example}
    In the above example of independent Bernoulli trials and the pattern of interest being a run of 1's of length $k \sim a \log n$, the bound in  Corollary \ref{cor:seq_patt} is of order $(\log n)^{-1}.$ 
\end{example}

\begin{proof}
We couple $U_n = \frac1n \max\{V_1, \cdots, V_M\}$ and $Z_n = \frac1n \max\{T_1, \cdots, T_M\}$, where $T_i \sim \text{Geometric}(\frac{\lambda}{n})$ for $i = 1, \ldots, M$, by using the same random variable $M$. Then 
\begin{equation} \label{sq1}
    d_{BW}(\mathcal{L}(U_n), \mathcal{L}(PE(\theta, \lambda))) \le  d_{BW}(\mathcal{L}(U_n), \mathcal{L}(Z_n)) + d_{BW}(\mathcal{L}(Z_n), \mathcal{L} (PE(\theta, \lambda))).
\end{equation}
Taking a union bound,
\begin{align*}
    d_{TV}(\mathcal{L}(U_n), \mathcal{L}(Z_n)) & \le \sum_{m=1}^{\infty} \sum_{i=1}^{m} 
{d_{TV}\left(\mathcal{L}\left(\frac{V{_i}}{n}\right), \mathcal{L}\left(\frac{T{_i}}{n}\right) \right)} \, 
\mathbb{P}(M=m)\le \frac{\lambda(k-1)}{n} \E M,
\end{align*}
using Corollary 1 in \cite{pekoz1996stein} in the last step. With \eqref{eq:dtv},
\begin{equation} \label{dis_BW1}
    d_{BW}(\mathcal{L}(U_n), \mathcal{L}(Z_n))\le  2d_{TV}(\mathcal{L}(U_n), \mathcal{L}(Z_n)) \le \frac{2\lambda(k-1)}{n}\EE M.
\end{equation}
Now 
\begin{eqnarray}
    d_{BW}\left(\mathcal{L}(Z_n), PE(\theta, \lambda)\right) &\le& d_{BW}\left(\mathcal{L}(Z_n), \mathcal{L}\left(\frac{W_1(M', {\bf{T}})}{n}\right)\right) \label{eq:firstterm}  \\
    && + d_{BW}\left(\mathcal{L}\left(\frac{W_1(M', {\bf{T}})}{n}\right), PE(\theta, \lambda)\right). \label{eq:secondterm}
\end{eqnarray}
The term \eqref{eq:secondterm} is bounded in \eqref{PGDbound}. To bound \eqref{eq:firstterm}
\eqref{disBW_1_exp} gives 
\begin{eqnarray*}
    d_{BW}\left(\mathcal{L}(Z_n), \mathcal{L}\left(\frac{W_1(M', {\bf{T}})}{n}\right)\right) &\le &\sum_{m,m'=1}^\infty \PP (M=m, M'=m') \\
    && \E \left| W_1(m', {\bf{T}}) -
     W_1(m, {\bf{T}})\right|. 
\end{eqnarray*}
The expectation of the maximum of $m$ Geometric$(p)$  variables satisfies   
$$ \frac{1}{\ln(p^{-1})} \sum_{k=1}^m\frac1k \le \E W_1(m, {\bf{T}}) \le 1 + \frac{1}{\ln(p^{-1})} \sum_{k=1}^m\frac1k,$$
as given on page 136 of \cite{EISENBERG2008135}. Hence 
\begin{eqnarray*}
   \lefteqn{d_{BW}\left(\mathcal{L}(Z_n), \mathcal{L}\left(\frac{W_1(M', {\bf{T}})}{n}\right)\right) } \\
   &\le&  \sum_{m,m'=1}^\infty \PP (M=m, M'=m') \frac{1}{\ln(p^{-1})} \sum_{k= \min(m,m')}^{\max(m,m')}\frac1k \\
    &=&  \frac{1}{\ln(p^{-1})} \left| \E \sum_{k= 1}^{M}\frac1k - \E \sum_{k= 1}^{M'}\frac1k \right|.
\end{eqnarray*}
Combining this result with \eqref{PGDbound} and \eqref{dis_BW1} in \eqref{sq1} we obtain the assertion.
\end{proof}
\begin{remark}
For $M=M'$ the bound in Corollary \ref{cor:seq_patt} reduces to
    \begin{eqnarray*} 
    \lefteqn{ d_{BW}(\mathcal{L}(U_n), PE(\theta, \lambda)) \le \frac{2\theta\lambda(k-1)}{n(1-e^{-\theta})} + \frac{e^{\frac{\lambda}{n}}}{n} \left(1+\frac{e^{ \frac{\theta\lambda}{n e}}}{n}B_1(\theta, \lambda)\right)} \\
     &+& \frac{1}{n} \left(\left(1-\frac{\lambda}{n}\right)^{-6} e^{\theta e^{ \frac{\lambda}{n}} + 2\frac{\theta\lambda}{n} + \frac{\theta\lambda}{n e} + \frac{\lambda}{n}}B_2(\theta, \lambda) + \frac{1}{2n} e^{\frac{\lambda}{n} + \frac{\theta\lambda}{n e}}B_3(\theta, \lambda)\right).  
\end{eqnarray*}
\end{remark}
The assumption of i.i.d.\,sequences can be weakened to that of a Markov chain by applying Theorem 5.5 in \cite{reinert2000probabilistic}, with $M=M'$. This theorem gives a Poisson process approximation for the number of ``declumped'' counts of each pattern, which in turn yields that the waiting time for each pattern to occur is approximately exponentially distributed. The theorem also gives an explicit bound on the approximation, but requires considerable notation, and hence we do not pursue it here.

%%%%%%%%%%Declarations%%%%%%%%%%

{\bf Acknowledgements.} We thank  Christina Goldschmidt, David Steinsaltz and  Tadas Temcinas for 
helpful discussions. We would also like to thank the anonymous reviewers for their suggestions which have led to an overall improved paper.

{\bf Funding information.} AF is supported by the Commonwealth Scholarship Commission, United Kingdom, and in part by EPSRC grant EP/X002195/1. 
GR is supported in part by EPSRC grants EP/T018445/1, EP/R018472/1, EP/X002195/1 and  EP/Y028872/1.

{\bf Competing interests}  There were no competing interests to declare which arose during the preparation or publication process of this article.

\bibliographystyle{APT}

\bibliography{bibliography}

\bigskip
\appendix
\section{Further proofs} \label{sec:proof}

\medskip \noindent 
{\bf Proof of  \eqref{bound_mean_GPE}.}

For $X \sim \text{GPE}(\theta, \lambda, \beta)$ with $\lambda,\theta  >0$ and $\beta \ge 1$, we have  
$(1-e^{-\theta + \theta e^{-\lambda x}})^{\beta-1} \le 1$ and hence 
\begin{align*}
     \E(X)  &= \frac{\beta \theta \lambda }{(1 - e^{-\theta})^{\beta}} \int_0^{\infty} x e^{-\lambda x - \theta \beta e^{-\lambda x}} (1-e^{-\theta + \theta e^{-\lambda x}})^{\beta -1}\mathrm{d}x \\
     &\le \frac{\beta \theta \lambda}{(1 - e^{-\theta})^{\beta}}  \int_0^{\infty} x e^{ -\lambda x }\mathrm{d}x. 
\end{align*}
$\hfill \Box$

\bigskip \noindent 
{\bf {Proof of inequality \eqref{eq:boundcd}.}} 

For the proof we first note that, for $n \in \N$ and $|x| < n$ we have 
\begin{equation}
    e^x\left(1-\frac{x^2}{n}\right) \le \left(1+\frac{x}{n}\right)^n\mbox{ and  } 0 \le e^x - \left( 1 + \frac{x}{n}\right)^n \le \frac{x^2}{n} e^x, \label{eq:useful1}
\end{equation} 
and $x e^{-x} \le e^{-1}$ for $x > 0$.
Hence, for $0 < \lambda < n$ and $z > 0$, 
   \begin{align}
  0 \le e^{-\lambda z} 
      - \left(1-\frac{\lambda z }{n z } \right)^{nz} \le \frac{ (\lambda z)^2}{nz} e^{-\lambda z} &= \frac{ \lambda^2 z}{n} e^{-\lambda z} \label{eq:help2_1} \\
      &\le \frac{ \lambda}{ne}  \label{eq:help2}
 \end{align}
Also, we can write $e^{-\theta \left(1-\frac{\lambda}{n}\right)^{nz}} 
= e^{\theta \left[ e^{-\lambda z}- \left(1-\frac{\lambda}{n}\right)^{nz}\right]} e^{-\theta e^{-\lambda z}}$, and so
   \begin{equation} \label{exponent_new}
    e^{-\theta \left(1-\frac{\lambda}{n}\right)^{nz}} \le e^{ \frac{\theta\lambda}{ n e}} e^{-\theta e^{-\lambda z}} .
\end{equation}

Now, to bound $l = \tau (d - c - \frac1n c')$, we have 
$$\tau(z) c(z) = e^{\lambda z + \theta e^{-\lambda z}} c(z) = \frac{\lambda \theta}{n}, \quad \quad\quad
\frac{2}{n} \tau(z)c'(z) = \frac{2\lambda^2 \theta}{n^2} \left(\theta e^{-\lambda z} -1\right),$$
and
\begin{align*}
   e^{\lambda z + \theta e^{-\lambda z}} d(z) 
   =& e^{\lambda z + \theta e^{-\lambda z} - \theta\left( 1 - \frac{\lambda}{n}\right)^{nz}} \left( e^{\frac{\lambda \theta}{n} \left( 1 - \frac{\lambda}{n}\right)^{nz}} - 1 \right) \\
   =& e^{\lambda z + \theta e^{-\lambda z} - \theta\left( 1 - \frac{\lambda}{n}\right)^{nz}} \left( e^{\frac{\lambda \theta}{n} \left( 1 - \frac{\lambda}{n}\right)^{nz}} - 1  - \frac{\lambda \theta}{n}\left( 1 - \frac{\lambda}{n}\right)^{nz} \right)\\
   &+ e^{\lambda z + \theta e^{-\lambda z} - \theta\left( 1 - \frac{\lambda}{n}\right)^{nz}}\frac{\lambda \theta}{n}\left( 1 - \frac{\lambda}{n}\right)^{nz}.
\end{align*}
Thus, 
\begin{eqnarray}
l(z) &=& \frac{2\lambda^2 \theta}{n^2} \left(\theta e^{-\lambda z} -1\right) \nonumber \\
    &&+ e^{\lambda z + \theta \left(e^{-\lambda z} - \left( 1 - \frac{\lambda}{n}\right)^{nz}\right)} \left( e^{\frac{\lambda \theta}{n} \left( 1 - \frac{\lambda}{n}\right)^{nz}} - 1  - \frac{\lambda \theta}{n}\left( 1 - \frac{\lambda}{n}\right)^{nz} \theta\right) \label{eq:decterm1} \\
    &&+ \frac{\lambda \theta}{n} e^{\lambda z + \theta e^{-\lambda z}}\left( e^{ - \theta\left( 1 - \frac{\lambda}{n}\right)^{nz}}\left( 1 - \frac{\lambda}{n}\right)^{nz} - e^{-\lambda z - \theta e^{-\lambda z}}\right). \label{eq:decterm2}
\end{eqnarray}
Here
\begin{equation} \label{eq:help1}
    \frac{2\lambda^2 \theta}{n^2} \left(\theta e^{-\lambda z} -1\right) \le \frac{2\lambda^2 \theta}{n^2}\max(1,\theta).
\end{equation}

To bound \eqref{eq:decterm1}, we use \eqref{eq:help2} and series expansion, recalling $0 < \lambda < n$, to get 
\begin{align}
  & e^{\lambda z + \theta e^{-\lambda z} - \left( 1 - \frac{\lambda}{n}\right)^{nz}\theta} \left( e^{\frac{\lambda \theta}{n} \left( 1 - \frac{\lambda}{n}\right)^{nz}} - 1  - \frac{\lambda \theta}{n}\left( 1 - \frac{\lambda}{n}\right)^{nz} \theta\right)\\& 
  \le e^{\lambda z + {\frac{\theta \lambda}{n e}}}  \sum_{k=2}^\infty
 \frac{(\lambda \theta \left( 1 - \frac{\lambda}{n}\right)^{nz})^k}
 {n^k k!} \nonumber \\
  &\le e^{\lambda z + {\frac{\theta \lambda}{n e}}}  \sum_{k=2}^\infty
 \frac{(\lambda \theta)^k  e^{-\lambda z k} }
 {n^k k!}
 \le \frac{1}{n^2} e^{-\lambda z + {\frac{\theta\lambda}{n e}} + \lambda \theta}
  \le \frac{1}{n^2} e^{{\frac{\theta}{e}} + \lambda \theta}. \label{eq:help3}
\end{align}

To bound \eqref{eq:decterm2}, we first bound 
\begin{eqnarray}
    \lefteqn{e^{-\lambda z - \theta e^{-\lambda z}} - e^{ - \theta\left( 1 - \frac{\lambda}{n}\right)^{nz}}\left( 1 - \frac{\lambda}{n}\right)^{nz}} \nonumber \\
    &=&   \left( e^{-\lambda z} -  \left(1-\frac{\lambda}{n}\right)^{nz} \right) e^{-\theta \left(1-\frac{\lambda}{n}\right)^{nz}} + e^{-\lambda z -\theta e^{-\lambda z}} \left(1 - e^{\theta \{e^{-\lambda z} - \left(1-\frac{\lambda}{n}\right)^{nz}\} }\right) \nonumber \\
     &\le&  \left( e^{- \lambda z} -  \left(1-\frac{\lambda}{n}\right)^{nz} \right) e^{-\theta \left(1-\frac{\lambda}{n}\right)^{nz}} \nonumber \\
     && + \theta e^{- \lambda z -\theta e^{-\lambda z}}  \left(e^{-\lambda z} - \left(1-\frac{\lambda}{n}\right)^{nz}\right) e^{\theta \{e^{-\lambda z} - \left(1-\frac{\lambda}{n}\right)^{nz}\}} \nonumber \\
     &\le& \frac{\lambda^2 z}{n} e^{ \frac{\theta\lambda}{ n e}} e^{-\lambda z -\theta e^{-\lambda z}} +   \frac{\theta\lambda}{n e} e^{ \frac{\theta\lambda}{n e}} e^{- \lambda z -\theta e^{-\lambda z}}. \label{new_expon}
\end{eqnarray}
Here we used Property 4 in \cite{salas2012exponential}: for all $x>0$ we have
 $ \left(1+\frac{x}{n}\right)^n - 1 \le xe^{x} $ and $ e^x - 1 \le x e^x$ along with \eqref{eq:help2_1} and \eqref{exponent_new}. Thus
\begin{equation}
    \frac{\lambda \theta}{n} e^{\lambda z + \theta e^{-\lambda z}}\left( e^{ - \theta\left( 1 - \frac{\lambda}{n}\right)^{nz}}\left( 1 - \frac{\lambda}{n}\right)^{nz} - e^{-\lambda z - \theta e^{-\lambda z}}\right) \le \frac{\theta\lambda^3z}{n^2} e^{ \frac{\theta\lambda}{n e}} +   \frac{\theta^2\lambda^2}{n^2 } e^{ \frac{\theta\lambda}{n e}-1}. \label{eq:help4}
\end{equation} 
Combining \eqref{eq:help1}, \eqref{eq:help3} and \eqref{eq:help4} we get
\begin{eqnarray*}
   l(z) \le  \frac{2\lambda^2 \theta}{n^2}\max(1,\theta) + \frac{1}{n^2} e^{\lambda \theta + {\frac{\theta}{e}}} + \frac{\theta^2\lambda^2}{n^2 } e^{ \frac{\theta\lambda}{n e}-1} + \frac{\theta\lambda^3z}{n^2} e^{ \frac{\theta\lambda}{n e}}.
\end{eqnarray*}
Replacing $z$ by $Z_n$ gives \eqref{eq:boundcd}. $\hfill \Box$

\bigskip \noindent 
{\bf {Bounding \eqref{appendix_A}.}}
For $z>0$,
\begin{eqnarray}
\lefteqn{2\frac{\lambda^2 \theta}{n^2}e^{-\lambda z - \theta e^{-\lambda z}}(\theta e^{-\lambda z}-1) } \nonumber\\
  &&- \left(e^{-\theta(1-\frac{\lambda}{n})^{nz+1}}- e^{-\theta(1-\frac{\lambda}{n})^{nz}}- e^{-\theta(1-\frac{\lambda}{n})^{nz-1}}+ e^{-\theta(1-\frac{\lambda}{n})^{nz-2}}\right) \nonumber \\
  &=&  2\frac{\lambda^2 \theta}{n^2}e^{-\lambda z - \theta e^{-\lambda z}}(\theta e^{-\lambda z}-1) - k\left( \frac{n}{n-\lambda} \right)  \nonumber 
\end{eqnarray}
where, with $a = \theta \left(1-\frac{\lambda}{n} \right)^{nz} $,
\begin{align*}
     k(b) = e^{-\frac{a}{b}} - e^{-a} - e^{-ab} + e^{-ab^2} = 2 (b-1)^2  ae^{-a} (a-1) + R_1 
\end{align*} 
and $R_1$ is the remainder term from Taylor expansion for $k(b)$ around 1;
\begin{align*}
    R_1 & = \frac{1}{3} (b-1)^3 a \left(a^2e^{-a\xi} + \frac{a^2 e^{-\frac{a}{\xi}}}{\xi^6} -\frac{6a e^{-\frac{a}{\xi}}}{\xi^5} + \frac{6 e^{-\frac{a}{\xi}}}{\xi^4} + 12 a \xi e^{-a\xi^2} -8a^2\xi^3 e^{-a\xi^2} \right). 
    \end{align*}
for some  $1 < \xi \le b = \frac{n}{n-\lambda}$. To bound $R_1$ we use that $1<\xi \le \frac{n}{n-\lambda}$ so that 
$e^{-a \xi} \le e^{-a}$ and $e^{-a \xi^2} \le e^{-a};$ also $e^{-\frac{a}{\xi}} \le e^{-a \frac{n-\lambda}{n}}.$
Hence with the crude bounds $a \le \theta,$ and $e^{-a} \le 1,$
\begin{align*}
   | R_1  | &\le  \frac{1}{3} (b-1)^3 a \left(a^2
    e^{-a} + \left( \frac{n-\lambda}{n} \right)^6 a^2 e^{-a \frac{n-\lambda}{n}}  + \left( \frac{n-\lambda}{n} \right)^5 6a e^{-a \frac{n-\lambda}{n}} \right.\\
     & \left.   +  \left( \frac{n-\lambda}{n} \right)^4 6 e^{-a \frac{n-\lambda}{n}} + 12 a\left(\frac{n}{n-\lambda}\right) e^{-a} + 8a^2 \left(\frac{n}{n-\lambda}\right)^3 e^{-a}  \right) \\
     & \le \frac{1}{3} (b-1)^3 a^2 e^{-a} \left(\theta +  \theta e^{\theta \frac{\lambda}{n}}  +  6 e^{\theta \frac{\lambda}{n}} +   6 a^{-1} e^{\theta \frac{\lambda}{n}} + 12 \left(\frac{n}{n-\lambda}\right)  + 8\theta \left(\frac{n}{n-\lambda}\right)^3  \right).
    \end{align*}
Substituting the expressions for $a$ and $b$,
\begin{align*}
     |R_1| & \le  \frac{1}{3} \left(\frac{\lambda}{n-\lambda}\right)^3 \theta^2 \left(1-\frac{\lambda}{n}\right)^{2nz} e^{-\theta \left(1-\frac{\lambda}{n}\right)^{nz}} \bigg[\theta + \theta e^{\frac{\lambda \theta}{n}} + 6 e^{\frac{\lambda \theta}{n}}  \\
    & \quad \quad \quad + 6 e^{\frac{\lambda \theta}{n}}\frac{1}{\theta}\left(1-\frac{\lambda}{n}\right)^{-nz}+ 12 \left(\frac{n}{n-\lambda}\right) + 8\theta \left(\frac{n}{n-\lambda}\right)^3\bigg]  \\
    & \le \frac{\theta^2\lambda^3}{3n^3} \left(1-\frac{\lambda}{n}\right)^{-3}   e^{-2\lambda z} e^{-\theta \left(1-\frac{\lambda}{n}\right)^{nz}} \bigg[\theta + \theta e^{\frac{\lambda \theta}{n}} + 6 e^{\frac{\lambda \theta}{n}} + 12 \left(1-\frac{\lambda}{n}\right)^{-1} \\
    & \quad \quad \quad   + 8\theta \left(1-\frac{\lambda}{n}\right)^{-3}\bigg] + \frac{2\theta\lambda^3}{n^3} \left(1-\frac{\lambda}{n}\right)^{-3} e^{\frac{\lambda \theta}{n}} e^{-\lambda z} e^{-\theta \left(1-\frac{\lambda}{n}\right)^{nz}}. 
\end{align*}
Here we used $e^{-\theta \left(1-\frac{\lambda}{n}\right)^{nz}} = e^{\theta \left[ e^{-\lambda z}- \left(1-\frac{\lambda}{n}\right)^{nz}\right]} e^{-\theta e^{-\lambda z}}.$ Hence, with \eqref{exponent_new},
\begin{eqnarray} \label{upperboundR}
    |R_1| &\le&  \frac{\theta^2\lambda^3}{3n^3}e^{{\frac{\theta \lambda}{n e}}} \left(1-\frac{\lambda}{n}\right)^{-3}   \bigg[\theta + \theta e^{\frac{\lambda \theta}{n}} + 6 e^{\frac{\lambda \theta}{n}} + 12 \left(1-\frac{\lambda}{n}\right)^{-1}  \\
    && + 8\theta \left(1-\frac{\lambda}{n}\right)^{-3}\bigg] e^{-2\lambda z-\theta e^{-\lambda z}} + 2\frac{\theta\lambda^3}{ n^3} \left(1-\frac{\lambda}{n}\right)^{-3}e^{ {\frac{\theta\lambda}{n e} (e+1)}} e^{-\lambda z-\theta e^{-\lambda z}
    }. \nonumber 
\end{eqnarray}
Thus,
\begin{eqnarray}
    \lefteqn{\bigg| 2\frac{\lambda^2 \theta}{n^2}e^{-\lambda z - \theta e^{-\lambda z}}(\theta e^{-\lambda z}-1) } \nonumber \\
    && -   (e^{-\theta(1-\frac{\lambda}{n})^{nz+1}}- e^{-\theta(1-\frac{\lambda}{n})^{nz}} - e^{-\theta(1-\frac{\lambda}{n})^{nz-1}}+ e^{-\theta(1-\frac{\lambda}{n})^{nz-2}})\bigg|\nonumber \\
    &\le& \bigg| 2\frac{\lambda^2 \theta}{n^2}e^{-\lambda z - \theta e^{-\lambda z}}(\theta e^{-\lambda z}-1) \bigg[1-\left(1-\frac{\lambda}{n}\right)^{-2} \bigg] \bigg| + |R_1| + |R_2| \label{diffpartii}
\end{eqnarray}
with 
\begin{align*}
   R_2 = & 2\left(\frac{\lambda}{n-\lambda}\right)^2 \bigg\{\theta^2 \bigg[\left(1-\frac{\lambda}{n}\right)^{2nz} e^{-\theta \left(1-\frac{\lambda}{n}\right)^{nz}} -e^{-2\lambda z-\theta e^{-\lambda z}} \bigg] \\
   & -  \theta \bigg[\left(1-\frac{\lambda}{n}\right)^{nz} e^{-\theta \left(1-\frac{\lambda}{n}\right)^{nz}} - e^{-\lambda z-\theta e^{-\lambda z}} \bigg] \bigg\}.
\end{align*}

Using \eqref{new_expon} we bound
\begin{align}
  |R_2| & \le 4\frac{ \theta^2 \lambda^4}{n^3} \left( 1- \frac{\lambda}{n}\right)^{-2} z e^{{\frac{\theta \lambda}{n e}}}  e^{-2\lambda z -\theta e^{-\lambda z}} +   2\frac{\theta^3 \lambda^3}{n^3 {e}} \left( 1- \frac{\lambda}{n}\right)^{-2}e^{{\frac{\theta\lambda}{n e}}}  e^{-2 \lambda z -\theta e^{-\lambda z}} \nonumber \\ & + 2\frac{ \theta \lambda^4 }{n^3}\left( 1- \frac{\lambda}{n}\right)^{-2} z e^{{\frac{\theta \lambda}{n e}}} e^{-\lambda z -\theta e^{-\lambda z}} +  2\frac{\theta^2 \lambda^3}{n^3 {e}} \left( 1- \frac{\lambda}{n}\right)^{-2} e^{{\frac{\theta\lambda}{n e}}} e^{- \lambda z -\theta e^{-\lambda z}}. \label{upperboundR2}
\end{align}
Combining \eqref{upperboundR}, \eqref{diffpartii} and \eqref{upperboundR2} and simplifying gives the bound 
\begin{align*}
 &{\bigg|2\frac{\lambda^2 \theta}{n^2}e^{-\lambda z - \theta e^{-\lambda z}}(\theta e^{-\lambda z}-1)} \nonumber \\
    &\quad \quad - \left(e^{-\theta(1-\frac{\lambda}{n})^{nz+1}}- e^{-\theta(1-\frac{\lambda}{n})^{nz}} - e^{-\theta(1-\frac{\lambda}{n})^{nz-1}}+ e^{-\theta(1-\frac{\lambda}{n})^{nz-2}}\right)\bigg|\nonumber \\ \label{diff} 
    & \le \frac{2\lambda}{n} \max(1,\theta) \left\{ \left(1-\frac{\lambda}{n}\right)^{-2} -1\right\} c(z) + \frac{\theta\lambda^2}{n^2} \left(1-\frac{\lambda}{n}\right)^{-2}  e^{{\frac{\theta\lambda}{n e}}} \bigg\{ \frac{1}{3} \left(1-\frac{\lambda}{n}\right)^{-1} \nonumber \\
    &  \left[\theta + \theta e^{\frac{\theta\lambda}{n}} + 6 e^{\frac{\theta\lambda}{n}} + 12 \left(1-\frac{\lambda}{n}\right)^{-1} + 8\theta \left(1-\frac{\lambda}{n}\right)^{-3}\right]   + 2\left({\frac{\theta}{e}}  + 2\lambda z \right)\bigg\} e^{-\lambda z}c(z) \nonumber \\ 
    &+ \frac{2\lambda^2}{n^2}\left( 1- \frac{\lambda}{n}\right)^{-2}e^{{\frac{\theta\lambda}{n e}}}\left[ \left(1-\frac{\lambda}{n}\right)^{-1}e^{\frac{\theta \lambda}{n}}  + \lambda z + {\frac{\theta}{e}} \right] c(z).
\end{align*}
Multiplying the above inequality by $n\left|g\left(z-\frac1n\right)\right|$ gives \eqref{appendix_A}. $\hfill \Box$

\end{document}